\documentclass[11pt,leqno]{article}

\usepackage{amsthm,amsfonts,amssymb,amsmath,oldgerm,color,bm,multicol}

\usepackage{amsthm,amsfonts,amssymb,amsmath,color}

\numberwithin{equation}{section}

\renewcommand\d{\partial}
\renewcommand\a{\alpha}

\renewcommand\o{\omega}

\newcommand\R{\mathbb R}\newcommand\N{\mathbb N}\newcommand\Z{\mathbb Z}
\newcommand\C{\mathbb C}

\def\th{\theta}

\def\l{\lambda}

\def\epsilon{\varepsilon}
\def\e{\varepsilon}

\def\dd{{\rm d}}

\newcommand\br{\begin{rem}}
\newcommand\er{\end{rem}}
\newcommand\bp{\begin{pmatrix}}
\newcommand\ep{\end{pmatrix}}
\newcommand\be{\begin{equation}}
\newcommand\ee{\end{equation}}
\newcommand\ba{\begin{equation}\begin{aligned}}
\newcommand\ea{\end{aligned}\end{equation}}

\newcommand\nn{\nonumber}

\newcommand{\calH}{\mathcal{H}}

\newcommand{\calP}{\mathcal{P}}

\newcommand{\dx}{{\rm d} {x}}

\newcommand{\tds}{{\tilde s}}

\newtheorem{defi}{Definition}[section]
\newtheorem{theorem}[defi]{Theorem}
\newtheorem{proposition}[defi]{Proposition}
\newtheorem{lemma}[defi]{Lemma}
\newtheorem{corollary}[defi]{Corollary}
\newtheorem{remark}[defi]{Remark}

\numberwithin{equation}{section}

\begin{document}

\title{Convergence rates in the nonrelativistic limit of the cubic Klein-Gordon equation}

\author{Weizhu Bao\footnote{Department of Mathematics, National University of Singapore, Singapore, 119076. Email: matbaowz@nus.edu.sg.}
\and Yong Lu\footnote{Department of Mathematics, Nanjing University,  210093 Nanjing, China. Email: luyong@nju.edu.cn. }
\and Zhifei Zhang\footnote{School of Mathematical Sciences, Peking University, 100871 Beijing, China. Email:  zfzhang@math.pku.edu.cn.}}

\date{} 

\maketitle

\begin{abstract}

In this paper, we study the nonrelativistic  limit of the cubic nonlinear Klein-Gordon equation in $\R^{3}$ with a small parameter $0<\e\ll 1$, which is inversely proportional to the speed of light. We show that the cubic nonlinear Klein-Gordon equation converges to the cubic nonlinear Schr\"odinger equation with a convergence rate of order $\e^{2}$. In particular, for the defocusing case and smooth initial data, we prove error estimates of the form $(1+t)\e^{2}$ at time $t$ which is valid up to long time of order $\e^{-1}$; while for nonsmooth initial data, we prove error estimates of the form $(1+t)\e$ at time $t$ which is valid up to long time of order $\e^{-\frac{1}{2}}$. These specific forms of error estimates  coincide with the numerical results obtained in \cite{BZ19,SZ20}.

\end{abstract}

{\bf Keywords}:
cubic Klein-Gordon equation; nonrelativistic limit; convergence rates.


\renewcommand{\refname}{References}

\section{Introduction}

\subsection{Setting}\label{sec:setting}

The Klein-Gordon equation is a relativistic version of the Schr\"odinger equation and is used to describe the
motion of a spinless particle. Consider the nonlinear Klein-Gordon equation in dimensionless form as
\be\label{0}
\e^2\d_{tt} u_{\e}-\Delta  u_{\e} + \e^{-2}u_{\e}  + f(u_{\e})=0,\quad t> 0,\quad x\in \R^3.
\ee
Here $u_{\e}:=u_{\e}(t,x)$ is the wave function (or order parameter) and is assumed to be real-valued, where the spatial variable $x$ is defined in the whole three dimensional space. We restrict to the typical cubic nonlinearity of which the structure is frequently used (see Remark \ref{rem-nonlinear} for more details):
\be\label{form-f}
f(u_{\e})=\l u_{\e}^{3}, \qquad \hbox{with}\quad  0\neq \l \in \R.
\ee
The non-dimensional small parameter $\e\in (0,1)$ is inversely proportional to the speed of light. In this paper, we focus on the convergence rates in the  nonrelativistic  regime ($\e \to 0$) of the nonlinear Klein-Gordon equation \eqref{0}. The initial datum is given by
\be\label{ini-0}
 u_{\e}(0,x)= \phi(x),\quad (\d_t u_{\e})(0,x)= \e^{-2} \psi(x),\quad x\in \R^3,
\ee
where $\phi$ and $\psi$ are real valued functions satisfying certain regularity.

For a fixed $\e>0$, the well-posedness of the nonlinear Klein-Gordon equation is well studied, see for example \cite{GV1, GV2}. The nonrelativistic limit of \eqref{0}-\eqref{ini-0} has gained a lot of interest both in mathematical analysis and in numerical analysis, see for example \cite{MS,MT, BN, M, Nak02, MN1, MN2, SV, BCZ14, BD12, BZ16,BZ19,BZ19B, SZ20} and references therein. In particular, for the typical polynomial nonlinearity $f(u)=\lambda |u|^q u$ with $0< q<\frac 4 {d-2}$,  Masmoudi and Nakanishi \cite{MN2} showed that a wide class of solutions $u$ to \eqref{0}-\eqref{ini-0} can be described by a system of coupled nonlinear Schr\"odinger equations. More precisely in \cite{MN2}, for $H^1$ initial data, a first order approximation and a convergence rate of order $o(\e^{\frac 12})$ was given;  for  $H^3$ initial data, a second order approximation and a convergence rate of order $o(\e)$ was given.  In their series of paper, Masmoudi and Nakanishi also studied the singular limits including nonrelativistic limit for related models and established rather complete theory for the asymptotic analysis in energy spaces: the limit from Klein-Gordon-Zakharov system to Zakharov system \cite{MN4, MN5, MN6}, from Maxwell-Klein-Gordon and Maxwell-Dirac to Poisson-Schr\"odinger \cite{MN3}.  In \cite{LZ17}, the second author and third author improved the error estimates in \cite{MN2} for smooth data:  $O(\e)$ for first order approximation and $O(\e^2)$ for second order approximation. Moreover, it is shown in \cite{LZ16} that for quadratic nonlinearity, the approximation given  in \cite{LZ17} is valid up to long time of order $\e^{-1}$.

Recently, based on extensive numerical results, it was observed numerically that
the solution of the cubic Klein-Gordon equation converges (i) of order
$O(\e^{2})$ with linear growth in time $t$ to the
solution of the cubic Schr\"odinger equation
for $H^3$ initial data \cite{BZ19}, and (ii) of order
$O(\e^{2})$ with global-in-time (or uniformly in time) to the
solution of the cubic Schr\"odinger equation with wave operator
for $H^2$ initial data \cite{BZ19}. Similar results were given in
\cite{BZ17} for the relativistic limit of the Klein-Gordon-Schr\"{o}dinger
system and in \cite{BZ16} for the high-plasma-frequency limit of
Klein-Gordon-Zakharov system. In particular, numerical simulations  in \cite{SZ20} and \cite{BZ16} indicate that, in the nonrelativistic regime, the Klein-Gordon(-Zakharov) equation converges to the nonlinear Schr\"odinger (Zakharov) equation with an error bound of the form $(1+t) \e^{2}$ for smooth data and of the form $(1+t) \e$ for nonsmooth data. Furthermore, Schwartz and Zhao \cite{SZ20} gave a priori error estimates of order  $O(\e^{2})$ for the first order expansion by numerical analysis. In this paper, we show exactly the forms of error estimates suggested by the numerical simulations for the defocusing case in three dimensional setting, i.e., $(1+t) \e^{2}$ for smooth data and of the form $(1+t) \e$ for nonsmooth data.

Throughout the paper, $C$ denotes a generic constant independent of $(\phi,\psi)$, $C_{0}$ denotes a generic constant depending on the norms $\|(\phi, \psi)\|_{H^{s}\cap L^{1}}$ of initial datum $(\phi,\psi)$; $D(t)$ denotes a generic positive function that is continuous and increasing in $t$ and may depend on the initial datum $(\phi,\psi)$ in terms of Sobolev norms $\|(\phi, \psi)\|_{H^{s}}$.  In particular,  $C, C_{0}$ and $D(t)$ are independent of $\e$, while the values of them may differ from line to line.  We often use $d$ to denote the spatial dimensions, while in this paper it is always taken $d=3$. We often write the Sobolev spaces $H^{s}(\R^{3})$ as $H^{s}$ for short.  We use $\Z_{\rm even}$ ($\Z_{\rm odd}$) to denote the collection of all even (odd) integers, and we use $\Z_+$ ($\Z_{\geq 0}$) to denote the collection of all positive (nonnegative) integers.

\subsection{Main results}\label{sec:intr-f-gene}

We will show that in the  nonrelativistic regime $\e \to 0$, the cubic Klein-Gordon equation can be well approximated by the cubic Schr\"odinger equation:
\be\label{eq-g-cub}
2i \d_{t} g_{0} -\Delta g_{0} +3\l |g_{0}|^{2} g_{0} =0,\quad g_{0}(0,\cdot)=\frac{\phi-i \psi}{2}.
\ee
We shall impose the following regularity assumptions on the initial data:
\be\label{ini-ass1}
\phi, \psi \in H^{s}(\R^{3}), \quad s\geq 1.
\ee
To ensure the decay estimates to the solutions of the cubic Schr\"odinger equation
\eqref{eq-g-cub}, we further assume
\be\label{ini-ass2}
\phi, \psi \in  L^{1}(\R^{3}).
\ee

We recall some results concerning the well-posedness of \eqref{eq-g-cub}. Firstly, the classical theory on hyperbolic equations gives (see Theorem 1 in \cite{book-Bour-Sch}, see also \cite{Met08}):
\begin{proposition}[Local well-posedness]\label{prop-Sch-local}
Suppose $s > \frac{3}{2}$ in \eqref{ini-ass1}.  Then there exists a unique local-in-time solution $g_{0} \in C([0,T^{*}), H^{s}(\R^{3}))$  to \eqref{eq-g-cub} with existence time $T^{*} = T^{*}(\|(\phi, \psi) \|_{H^{s}(\R^{3})}).$

\end{proposition}

For the focusing case $\l<0$, the solution of the cubic Schr\"odinger equation
\eqref{eq-g-cub} may blow up in finite time, while for the defocusing case with $\l>0$, the cubic Schr\"odinger equation is globally well-posed and enjoys certain large time decay estimates (see Theorem 2 and Remark 2 in \cite{book-Bour-Sch} for the global well-posedness and large time behavior (see also \cite{CKSTT04}, \cite{CKSTT08} and \cite{KM10}), see Theorem 1.1 in \cite{FSZ22}  and Theorem 1.4 in \cite{FZ20} for the decay estimates):
\begin{proposition}[Global well-posedness]\label{prop-Sch-global}
Let $\l>0$. Suppose $(\phi, \psi)$ satisfy \eqref{ini-ass1} with $s \geq 2$.  Then there exists a unique global-in-time solution $g_{0} \in L^{\infty}([0, \infty), H^{s}(\R^{3}))$ to \eqref{eq-g-cub} with
\be\label{est-Sch-global1}
\|g_{0}(t, \cdot)\|_{H^{s}(\R^{3})} \leq C_{0},  \quad   t\in [0,\infty).
\ee
Moreover, if the initial datum $(\phi, \psi)$ satisfies \eqref{ini-ass2}, there holds the decay estimate
\be\label{est-Sch-global2}
\|g_{0}(t, \cdot)\|_{L^{\infty}(\R^{3})} \leq  C_0(1+t)^{-\frac{3}{2}}, \quad   t\in (0,\infty).
\ee

\end{proposition}

\begin{remark}\label{rem-defSch}
The constant $C_{0}$ on the right-hand side of \eqref{est-Sch-global1} only depends on the norms $\|(\phi ,\psi)\|_{H^{s}(\R^{3})}$. As shown in Theorem 1.1 in \cite{FSZ22}, the constant $C_{0}$ on the right-hand side of \eqref{est-Sch-global2} depends on $ \|(\phi, \psi)\|_{H^{1} \cap L^{1} (\R^{3})}.$

The decay estimate \eqref{est-Sch-global2} can also be obtained by imposing different regularity assumptions on initial data. See for example Theorems 6.1 and 6.2 in \cite{HT86} where  the following weighted integrability conditions are imposed:
\be\label{ini-ass3}
\phi, \psi \in  W^{1,\frac{4}{3}}(\R^{3}), \quad |x|\phi, |x|\psi \in  H^{1}(\R^{3}) \cap L^{\frac{4}{3}}(\R^{3}).
\ee

\end{remark}

Based on the numerical results, it is observed in \cite{BZ19,SZ20} (and \cite{BZ16} for the Klein-Gordon-Zakharov equations and \cite{BZ17} for the Klein-Gordon-Sch\"{o}dinger equations) that, for smooth data, the convergence rates of the cubic Klein-Gordon equation to the cubic Schr\"odinger equation are of the form $(1+t)\e^{2}$ at time $t$ in the nonrelativistic regime $\e\to 0$, while for nonsmooth data, the convergence rates are of the form $(1+t)\e$.  We rigorously justify such numerical results in the defocusing case ($\l>0$) and we further show such a convergence rate is valid up to long time of order $\frac{1}{\e}$ for the case with smooth data and of order $\frac{1}{\sqrt\e}$ for the case with nonsmooth data. This is our first result:
\begin{theorem}\label{thm2}
 Let $\l>0$. Suppose $(\phi, \psi)$ satisfy \eqref{ini-ass1}--\eqref{ini-ass2}. 
 \begin{itemize}
 \item[(i)] If $s > 8 + \frac{3}{2}$, then \eqref{0}--\eqref{ini-0}  admits a unique solution $u_{\e}\in C([0,\frac{T_{0}}{\e}]; H^{s-8}(\R^{3}))$ for some $T_{0}>0$ independent of $\e$. Moreover, there holds the error estimate
\ba\label{stable-defo-thm2-3d}
\|\big (u_{\e}-(e^{i\frac{t}{\e^{2}}} g_{0} + e^{-i\frac{t}{\e^{2}}} \bar g_{0}) \big)(t,\cdot)\|_{H^{s-8}(\R^{3})} &\leq C_0(1+t)\e^{2}, \   \mbox{for all $t \leq \frac{T_{0}}{\e}$}.
\ea

\item[(ii)] If $s > 4 + \frac{3}{2}$,  then \eqref{0}--\eqref{ini-0}  admits a unique solution $u_{\e} \in C([0,\frac{\hat T_{0}}{\sqrt\e}]; H^{s-4}(\R^{3}))$ for some $\hat T_{0}>0$ independent of $\e$. Moreover, there holds the error estimate
\ba\label{stable-defo-thm3-3d}
\|\big (u_{\e}-(e^{i\frac{t}{\e^{2}}} g_{0} + e^{-i \frac{t}{\e^{2}}} \bar g_{0}) \big)(t,\cdot)\|_{H^{s-4}(\R^{3})} &\leq  C_0(1+t)\e, \    \mbox{for all $t \leq \frac{\hat T_{0}}{\sqrt\e}$}.
\ea
\end{itemize}
\end{theorem}

In  \cite{BZ19B} and \cite{SZ20},  a uniform second order approximation was shown by numerical study where the error estimates are of order $O(\e^{4})$.  We can actually give an expansion of the solution $u_{\e}$ up to an arbitrary  order provided the initial data are smooth enough. This is the following theorem:
\begin{theorem}\label{thm4}
 Let $K_{a}\in \Z_{\geq 0} \cap \Z_{\rm even}$. Suppose $(\phi, \psi)$ satisfy \eqref{ini-ass1}--\eqref{ini-ass2} with $s> s_{a}:= 2  K_{a} +4 + \frac{3}{2}$. 
  \begin{itemize}
 \item[(i)] If $\l >0$ and $K_{a} >0$,  then \eqref{0}--\eqref{ini-0}  admits a unique solution $u_{\e}\in C([0,\frac{T_{0}}{\e}]; H^{s-2K_{a}-4}(\R^{3}))$ with for some $T_{0}>0$ independent of $\e$. Moreover, there hold the estimates
\ba\label{stable-defo-thm4-3d}
\|(u_{\e}-u_{a})(t,\cdot)\|_{H^{s-2K_{a}-4}(\R^{3})} &\leq  C_0(1+t)^{K_{a}}\e^{K_{a}+1}, \   \mbox{for all $t \leq \frac{T_{0}}{\e}$}.
\ea
Here $u_{a}$ is of the form
\ba\label{ua-form-thm4}
u_{a} = u_{0} + \e^{2} u_{2} + \cdots + \e^{K_{a}} u_{K_{a}} +  \e^{K_{a}+2} u_{K_{a}+2},
\ea
where
\ba\label{unp-form-thm4}
u_{0} : = e^{i\frac{t}{\e^{2}} } g_{0} + e^{-i \frac{t}{\e^{2}} } \bar g_{0}, \quad u_{n} = \sum_{|p|\leq n+1} e^{ip \frac{t}{\e^{2}} } u_{n,p},
\ea
where $g_{0}$ is the solution to the cubic Schr\"odinger equation \eqref{eq-g-cub} satisfying the estimates in Proposition \ref{prop-Sch-global}, and there hold the following estimates for each $n=2,4,\cdots,K_{a}$, each $p,$ and for all $t\in (0,\infty)$:
\ba\label{est-ua-thm4}
 \|u_{n,p}(t,\cdot)\|_{H^{s-2n}(\R^{3})}  \leq  C_0(1+t)^{n-1} ,  \
\|u_{K_{a}+2,p}(t,\cdot)\|_{H^{s-2K_{a}-2}(\R^{3})}   \leq  C_0(1+t)^{K_{a}-1}.
\ea
\item[(ii)] If $\l <0$,  then Cauchy problem of the cubic Klein-Gordon equation \eqref{0}--\eqref{ini-0} admits a unique solution $u_{\e}\in C([0,T_{\e}^{*}); H^{s-2K_{a}-4}(\R^{3}))$ with the existence time $T_{\e}^{*}$ satisfying
\ba\label{exist-time-thm1}
\liminf_{\e \to 0} T^{*}_{\e} \geq T^{*},
\ea
where $T^{*}$ is the existence time to the Cauchy problem of the cubic Schr\"odinger equation \eqref{eq-g-cub} given in Proposition \ref{prop-Sch-local}.
Moreover, there holds the error estimate
\ba\label{stable-thm1}
\| (u_{\e}-u_{a})(t,\cdot) \|_{H^{s-2K_{a}-4}(\R^{3})} \leq D(t)\e^{K_{a}+1}, \quad   t <\min\{T^{*}, T_{\e}^{*}\}.
\ea
Here $u_{a}$ is of the same form as in \eqref{ua-form-thm4} and \eqref{unp-form-thm4} with the estimates for each  $n=2,4,\cdots,K_{a}$, for each $p$, and for all $t<\min\{T^{*}, T_{\e}^{*}\}$:
\ba\label{est-ua-thm1}
\|g_{0}(t,\cdot)\|_{H^{s}(\R^{3})} + \|u_{n,p}(t,\cdot)\|_{H^{s-2n}(\R^{3})} + \|u_{K_{a}+2,p}(t,\cdot)\|_{H^{s-2K_{a}-2}(\R^{3})}   \leq D  (t).
\ea
In particular, if $K_{a} \geq 2$, we have
\ba\label{stable-thm1-g0}
\| \big(u_{\e}-e^{i \frac{t}{\e^{2}} } g_{0} + e^{-i \frac{t}{\e^{2}} } \bar g_{0}\big)(t,\cdot) \|_{H^{s-2K_{a}-4}(\R^{3})} \leq D (t)\e^{2}, \quad   t <\min\{T^{*}, T_{\e}^{*}\}.
\ea

\end{itemize}
\end{theorem}
Recall that $D(t)$ denotes a positive increasing function in $t$.

\medskip

The rest of the paper is devoted to the proof of our main theorems. In the next section we will illustrate the main ideas and novoelties. In the sequel to avoid notation complexity, we will drop the subscript $\e$ in $u_{\e}$ and simply denote the solution of the Klein-Gordon equation by $u.$

\subsection{Main ideas}

The main novelty of our results is to derive specific growth rates in time for the error estimates in the nonrelativistic limit of the Klein-Gordon equations for certain cases. In particular, for the defocusing case with cubic nonlinearity, linear in time growth rates are obtained and are valid up to long time of order $\e^{-1}$ or $\e^{-\frac 12}$. Our idea is to transfer the study of the nonrelativistic limit problem to the stability problem of the WKB approximate solutions in geometric optics, see Section \ref{sec:reform}. The main theorems can be treated as corollaries of the stability results given in Section \ref{sec:reform}.

In the earlier works (see \cite{MN1, Nak02, M, BN, MT} and improved results in \cite{MN2}), only locally uniform in time error estimates are shown. The time dependency of the convergence rates is not clear, even for the defocusing case where the global-in-time energy estimates or even the scattering hold.   In \cite{MN2}, the convergences were derived in the $H^{1}$ energy space induced by the conserved quantities of the nonlinear Klein-Gordon equations, without assuming too much regularity for the initial data, say $H^{3}$ data. While, even assuming sufficient high regularity of initial data, it seems that the method in \cite{MN2} cannot be applied to deduce the error estimates grow at most algebraically in time.  The proof in \cite{MN2} is based on considering the new unknown
$$
v := e^{-\frac{it}{\e^{2}}} \frac{1}{2}\left( u - \frac{i\e }{\langle \nabla \rangle_{\e}} (\d_{t} u)\right), \quad \langle \nabla \rangle_{\e}: = \sqrt{\e^{-2} - \Delta},
$$
 where $u$ is the solution to the nonlinear Klein-Gordon equations and is assumed to be real valued here. Then $u$ can be written as
 $$
 u = e^{\frac{it}{\e^{2}}}  v + e^{-\frac{it}{\e^{2}}}  \bar v.
 $$
Moreover,  $v$ satisfies
 $$
 2i \d_{t} v + 2\e^{-1} (\langle \nabla \rangle_{\e} - \e^{-1}) v + \frac{1}{\e \langle \nabla \rangle_{\e}} \tilde f (v) = 0.
 $$
Here the nonlinear term $ \tilde f (v) = e^{-\frac{it}{\e^{2}}}  f(e^{\frac{it}{\e^{2}}}  v + e^{-\frac{it}{\e^{2}}}  \bar v)$.  With cubic nonlinearity \eqref{form-f}, one has
$$
\tilde f(v) = 3 \l |v|^{2} v + \l (e^{\frac{2it}{\e^{2}}} v^{3} +  2 e^{\frac{-2it}{\e^{2}}} |v|^{2}\bar v + e^{\frac{-4it}{\e^{2}}} \bar v^{3})
$$
With such reformulation, it is the key to show the convergence $v \to g_{0}$ as $\e \to 0$, where $g_{0}$ is the solution to the nonlinear Schr\"odinger equation \eqref{eq-g-cub}.  Observe
$$
2\e^{-1} (\langle \nabla \rangle_{\e} - \e^{-1}) = 2\frac{-\Delta} {\sqrt{1-\e^{2} \Delta} + 1} =  -\Delta + O(\e^{2}).
$$
Using Duhamel formula and employing Strichartz estimates, even with low regularity data, Masmoudi and Nakanishi \cite{MN2} were able to establish the convergence from $v$ to $g_{0}$ in energy space $H^{1}$ as $\e \to 0$ with convergence rate $o(\e^{-\frac 12})$. In their proof, both uniform $H^{1}$ norms of $v$ and $g_{0}$  are needed. The $H^{1}$ norm of $v$ comes from the $H^{1}$ norm of $u$.  However, one cannot obtain the uniform $H^{1}$ bound of $u$ as $\e \to 0$ from the conservation of energy :
\ba
E_{\e}(u(t)) &: = \int_{\R^{d}} \left( |\e \d_{t} u |^{2} + |\nabla u|^{2} + |\e^{-1} u|^{2} + \l |u|^{4} \right) \dx \\
 & =  E_{\e}(u(0)) := \e^{-2} \int_{\R^{d}} \left( |\psi |^{2} + |\phi|^{2} \right)\dx  +\int_{\R^{d}}\left( |\nabla u|^{2} + \l |u|^{4} \right) \dx,
\nn\ea
where one sees $\|u(t, \cdot)\|_{H^{1}}$ may blow up as $\e \to 0$.  With certain additional assumptions on the nonlinearities, for example with defocusing cubic nonlinearity where global well-posedness and scattering holds,  uniform $H^{1}$ bound of $u$ as $\e \to 0$ can be obtained (see Theorem 1.5 in \cite{MN2}):  $\|u(t, \cdot)\|_{H^{1}} \leq D(t)$ for some function $D : [0,\infty) \to  [0,\infty)$. However, the behavior of $D$ as $t\to \infty$ is not clear, no matter how smooth the initial data are.

\medskip

In our proof, we do not need any a priori uniform estimates for the exact solution $u$. We need only the estimates of the solutions $g_{n}$ to the Schr\"odinger equations which are independent of $\e$ and are well studied with rich results.   After the construction of an approximate solution $u_{a}$ via WKB expansion, we turn to study the equation of the perturbation $\dot u : = \e^{-2}(u - u_{a})$.  There are mainly three key ingredients in the derivation of specific linear-in-time growth convergence rates over long time of order $\e^{-1}$ or $\e^{-\frac 12}$:

The first one is to derive specific growth rates in time for the terms in the approximate expansion of $u$, see \eqref{est-ua-thm4}. Each term in the approximate expansion of $u$ can be written in terms of $g_{n}$ being solutions of Schr\"odinger equations---nonlinear for the leading term $g_{0}$ and linear for the correctors $g_{n}, n\geq 1$. Thus, it is crucial to derive the lowest possible growth rates in time for each $g_{n}$.   In $3D$ setting, the cubic defocusing Schr\"odinger equation with regular initial data is globally well-posed with solution $g_{0}\in C([0,\infty);H^{s}(\R^{3})), s\geq 1$; moreover, the higher order Sobolev norms $\|g_{0}(t,\cdot)\|_{H^{s}(\R^{3})}, s >1$ are uniformly bounded with respect to time $t\in (0,\infty).$   The linear Schr\"odinger equations for the higher order terms $g_{n}, n\geq 1$ are of the form
\be
2 i \d_t   g_{n}  - \Delta  g_{n}  = - 3 \l ( g_{0}^{2} \bar g_{n} + 2 |g_{0}|^{2} g_{n}) + l.o.t.,
\ee
where $l.o.t.$ represent the sum of some lower order terms that are functions of $g_{k}, k\leq n-1$. Further observing the leading term $g_{0}$ admits decay estimates $\|g_{0}(t,\cdot)\|_{L^{\infty}(\R^{3})} \leq C_{0} (1+t)^{-\frac 32}$, one deduces
$$
\|(g_{0}^{2} \bar g_{n} + 2 |g_{0}|^{2} g_{n})(t, \cdot)\|_{H^{s}} \leq C \|g_{0}(t, \cdot)\|_{L^{\infty}} \|g_{0}(t, \cdot)\|_{H^{s}} \|g_{n}(t, \cdot)\|_{H^{s}} \leq C_{0}(1+t)^{-\frac 32} \|g_{n}(t, \cdot)\|_{H^{s}}.
$$
The integrability of $(1+t)^{-\frac 32}$ in $t\in (0,\infty)$ ensures $g_{n}$ has at most algebraic growth in time $t$ as $t\to \infty$. More precisely, we can show the higher order Sobolev norms of $g_{n}$ behave like $(1+t)^{n-1}$ as $t\to \infty.$

\medskip

The second one is to observe that the initial perturbation can be as made to be small of an arbitrary order $O(\e^{K_{a}+1})$ for any $K_{a} \in \Z_{+}$ as long as the initial data belong to $H^{s}$ with $s>2K_{a} + 4 + \frac{3}{2}$, please see more explanation in (3) in Section \ref{sec-keys}.

\medskip

The third one is to reveal the perturbed equation in $\dot U:=\e^{-2}(U - U_{a})$ of the form
$$
\d_{t} \dot U + \mathcal{A}_{\e} \dot U = L_{\e} (t) \dot U +  \e^{2}  F_{\e} (\dot U)+ R_{\e},
$$
where $\mathcal{A}_{\e}$ is symmetric hyperbolic with $e^{-i\mathcal{A}_{\e}}$ a unitary from $H^{s}$ to $H^{s}$, $L_{\e}(t)$ is a linear operator and $F_{\e}(\dot U)$ is a polynomial in $\dot U$ of degree $3$. Here the new unknown $U$ contains the information of $u$ and $\e^{2} \d_{t} u$, see \eqref{new-var} and \eqref{def-dotU}. It is crucial to prove that the linear operator $L_{\e}(t)$ enjoys the following specific time growth (see Section \ref{sec:proof-defo}):
$$
\| L_{\e} (t)\|_{\mathcal{L}(H^{s}, H^{s})} \leq C_{0} (1+t)^{-\frac 32} +  C_{0}\e^{2} (1+t) \big(1 + \e (1+t) + \cdots +    \e^{2K_{a}} (1+t)^{2 K_{a}}\big),
$$
where the leading order term is integrable in $t$ over  $(0,+\infty)$. This allows us to employ the classical theory of symmetric hyperbolic equations to deduce the estimate $\|\dot U(t)\|_{H^{s}} \leq C_{0}(1+t)$ over long time of order $\e^{-1}$ and further shows that $ \| (U -U_{a})(t)\|_{H^{s}} \leq C_{0}(1+t)\e^{2}$.

\medskip

This paper is organized as follows. In Section \ref{sec:reform}, we introduce some new unknowns and rewrite the cubic Klein-Gordon equation as a symmetric hyperbolic system and then reformulate the  nonrelativistic  problem into the stability problem of the WKB solutions in nonlinear geometric optics.  In Sections \ref{sec:wkb}, we proceed the WKB expansion and give the specific structure of the WKB solution.  In Section \ref{sec:reg-wkb}, we show the regularity of the components in the WKB solution with specific growth in time estimates, and complete the construction of approximate solutions. In Section \ref{sec:sta-WKB}, we show the stability of the WKB solutions constructed in Section \ref{sec:wkb}. As corollaries of the stability of the WKB solutions,  we finally in Section \ref{sec:sta-KG-S} give the proofs of Theorems \ref{thm2} and \ref{thm4}.

\section{Reformulation via nonlinear geometric optics}\label{sec:reform}

In this section, we reformulate this nonrelativistic limit problem as the stability problem of WKB approximate solutions in nonlinear geometric optics.

\subsection{The equivalent symmetric hyperbolic system}\label{sec:eqhy}

We rewrite the cubic Klein-Gordon equation into a symmetric hyperbolic system by introducing
\be\label{new-var}
U:=(w^{\rm T},v, u)^{\rm T}  :=\left(\e(\d_{x_1}u,\cdots, \d_{x_d} u),\e^2\d_t u,u\right)^{\rm T}, \quad w:=\e\nabla u=\e(\d_{x_1}u,\cdots, \d_{x_d} u)^{\rm T}.
\ee
Then the equation \eqref{0} is equivalent to
\be\label{00}
\d_t U-\frac{1}{\e}A(\nabla)U +\frac{1}{\e^2} A_0 U = F(U),
\ee
where
\be\label{def:AB}
A(\nabla):=\bp 0_{d\times d} &\nabla & 0_d\\ \nabla^{\rm T} &0&0\\ 0_d^{\rm T}&0&0 \ep, \quad A_0:=\bp 0_{d\times d} &0_d & 0_d\\ 0_d^{\rm T}&0&1\\0_d^{\rm T}&-1&0 \ep,\quad F(U)=-\bp 0_d\\f(u)\\0\ep.
\ee
Here the notation $\nabla:= (\d_{x_1},\cdots,\d_{x_d})^{\rm T}$,  $0_{d\times d}$ denotes the zero matrix of order $d\times d$, and  $0_d$ denotes the zero column vector of dimension $d$. When there is no confusion from the context, we sometimes omit the subscript $d$ or $d\times d$ and simply use $0$. With initial datum \eqref{ini-0}, one has
\be\label{ini-data00}
U(0,\cdot)=\left(\e \nabla^{\rm T} \phi , \psi ,\phi \right)^{\rm T}.
\ee

\subsection{WKB solutions and stability}\label{sec:intr-wkb}

We look for an approximate solution to \eqref{00} by using WKB expansion which is a typical technique in geometric optics. The main idea is as follows.

We make a formal power series expansion in $\e$ for the solution and each term in the series is a trigonometric polynomial in $\th:=\frac{t}{\e^2}$:
 \be\label{def-wkb}
 U_a = \sum_{n=0}^{K_a+2} \e^{n} { U}_{n}, \qquad { U}_{n} = \sum_{p \in {\cal H}_n} e^{i p \th} U_{n,p}, \qquad K_a \in \Z_{\geq 0}, \,\, {\cal H}_n \subset \Z.\ee
 As in \eqref{new-var}, hereafter we use the notation for the corresponding components:
\be\label{Un-3n}
U_{a} =  \bp w_a\\ v_a\\ u_a\ep, \quad  U_n = \bp w_n\\ v_n\\ u_n\ep,\quad U_{n,p} = \bp w_{n,p}\\ v_{n,p}\\ u_{n,p}\ep.\nn
\ee

 The amplitudes $U_{n,p}(t,x)$  are not highly-oscillating (independent of $\th$) and satisfies $U_{n,-p}=\overline U_{n,p}$  due to the reality of $U_a$. Here ${\cal H}_n$ is the $n$-th order harmonics set which has finite elements and will be determined later on (see \eqref{Hn-form}) in the construction of $U_a$. The \emph{zero-order} or \emph{fundamental} harmonics set ${\cal H}_0$ is defined as ${\cal H}_0:=\{p\in \Z:~{\rm det}\,(ipI_{d+2}+A_0)=0\}$ with $I_{d+2}$ denoting the $(d+2)\times (d+2)$ identity matrix. For the nonhomogeneous case with $A_0\neq 0$, the set ${\cal H}_0$ is always finite. Indeed, with $A_0$ given in \eqref{def:AB}, we have ${\cal H}_0=\{-1,0,1\}.$ Higher order harmonics are generated by the fundamental harmonics and the nonlinearity of the system. In general  there holds the inclusion ${\cal H}_{n}\subset {\cal H}_{n+1}$ and ${\cal H}_{n+1}$ is generated by ${\cal H}_{n}$ through the nonlinearity of the equation. This will be more clear later in the WKB cascade in Section \ref{sec:wkb}.

\medskip

We plug \eqref{def-wkb} into \eqref{00} and deduce the system of order $O(\e^n)$:
\ba\label{phi-n-p}
\Phi_{n,p}:=\d_t U_{n,p}- A(\nabla) U_{n+1,p} +\big(ipI_{d+2} +A_0\big) U_{n+2,p} - F(U_a)_{n,p}=0,
\ea
for any {$n\in \Z,~n\geq -2$ and  $p\in \Z$}. Here
\ba\label{Fnp}
F(U_{a})_{n,p}:= \bp 0_d \\ - f(u_{a})_{n,p} \\ 0 \ep, \quad f(u_{a})_{n,p}: = \l  \sum_{n_{1}+ n_{2} + n_{3} = n}\; \sum_{p_{1}+p_{2} + p_{3} = p} u_{n_{1},p_{1}} u_{n_{3},p_{1}} u_{n_{3},p_{1}}.
\ea
Naturally, we always impose $U_{n}=0$ for $n = -2, -1$. We can solve \eqref{00} approximately by solving $\Phi_{n,p}=0$  up to some nonnegative order ${K_a}$, such that  $U_a$ solves \eqref{00} with a remainder of order $O(\e^{K_a+1})$ which is small and  goes to zero in the limit $\e \to 0$. Such $U_{a}$ is called a \emph{WKB approximate solution} or \emph{WKB solution} for short. Without loss of generality we shall choose $K_{a}\in \Z_{\rm even}$.

\medskip

We show that a WKB approximate solution up to an arbitrary order can be constructed. 
\begin{theorem}\label{thm:app-defo}
Let $K_{a} \in \Z_{\geq 0}\cap \Z_{\rm even}$. Suppose $(\phi, \psi)$ satisfy \eqref{ini-ass1}--\eqref{ini-ass2} with $s> s_{a}:= 2  K_{a} +4 + \frac{3}{2}$.  Let $T^{*}$ be the existence time to the Cauchy problem of the cubic Schr\"odinger equation \eqref{eq-g-cub} given in Proposition \ref{prop-Sch-local}. 
\begin{itemize}

\item[(i)] If $\l>0$, then $T^{*} = \infty$ and  there exists a WKB solution $U_{a}$ of the form \eqref{def-wkb} with $U_{n,p} \in C([0,\infty); H^{s-2n}(\R^{3}))$ satisfying the estimates in Corollary \ref {prop-reg-Un-defo}  for each $n=0,1,2, \cdots, K_{a}+2$ and each $p\in \calH_{n}$ given in \eqref{Hn-form}.   Moreover $U_{a}$ satisfies 
\be\label{eq-WKB-fo}
\left\{ \begin{aligned}
&\d_t U_a  - \frac{1}{\e}A(\nabla) U_a  +\frac{1}{\e^2} A_0 U_a  = F(U_a ) - \e^{K_{a}+1} R_\e ,\\
& U_a (0,\cdot) = U(0,\cdot)-\e^{K_{a}+2} r_\e ,\end{aligned}\right.
\ee
 for all $t\in (0,\infty)$ with
\be\label{est-remainders-defo}
\|R_{\e} (t,\cdot)\|_{ H^{s-2 K_{a}-4}} \leq  C_0(1+t) ^{K_{a}-1},   \quad \|r_{\e} \|_{ H^{s-2 K_{a}-2}} \leq  C_{0}.
\ee
  
 \item[(ii)] If $\l<0$, then there exists a WKB solution $U_a$ of the form \eqref{def-wkb} with $U_{n,p} \in C([0,T^{*}); H^{s-2n}(\R^{3}))$  for each $n=0,1,2, \cdots, K_{a}+2$ and each $p\in \calH_{n}$ given in \eqref{Hn-form}.  Moreover, $U_{a}$ solves \eqref{eq-WKB-fo} for all $t\in (0,T^{*})$ with
\be\label{est-remainders-fo}
\|R_{\e}(t,\cdot)\|_{H^{s-2 K_{a}-4}} \leq D(t), \quad \|r_{\e}\|_{ H^{s-2 K_{a}-2}} \leq C_{0}.
\ee

 \end{itemize}
\end{theorem}

We show the WKB approximate solutions given in Theorem \ref{thm:app-defo} are stable:
\begin{theorem}\label{thm:sta-defo}  Let $K_{a} \in \Z_{+}\cap \Z_{\rm even}$. Suppose $(\phi, \psi)$ satisfy \eqref{ini-ass1}--\eqref{ini-ass2} with $s> s_{a}:= 2  K_{a} +4 + \frac{3}{2}$.  Let $U$ be the unique solution up to certain existence time of the Cauchy problem of the cubic Klein-Gordon equation \eqref{00}--\eqref{ini-data00}.

\begin{itemize}

\item[(i)] If $\l>0$ and $K_{a}>0$, then $U\in C([0,\frac{T_{0}}{\e}]; H^{s-2K_{a}-4}(\R^{3}))$ for some $T_{0}>0$ independent of $\e$ and the WKB solution $U_{a}$ given in Theorem \ref{thm:app-defo} is stable in the following sense:
\ba\label{stable-defo-3d}
\|(U-U_{a})(t,\cdot)\|_{H^{s-2K_{a}-4}(\R^{3})} &\leq  C_0(1+t) ^{K_{a}}\e^{K_{a}+1}, \  \mbox{for all $t \leq \frac{T_{0}}{\e}$}.
\ea
\item[(ii)] If $\l>0$ and $K_{a} = 0$, then  $U\in C([0,\frac{\hat T_{0}}{\sqrt \e}]; H^{s-4}(\R^{3}))$ for some $\hat T_{0}>0$ independent of $\e$ and the WKB solution $U_{a}$ given in Theorem \ref{thm:app-defo} with $K_{a} = 0$ is stable in the following sense:
\ba\label{stable-defo-3d-0}
\|(U-U_{a})(t,\cdot)\|_{H^{s-4}(\R^{3})} &\leq  C_0(1+t) \e, \  \mbox{for all $t \leq \frac{\hat T_{0}}{\sqrt\e}$}.
\ea

\item[(iii)] If $\l<0$, then $U\in C([0,T_{\e}^{*}); H^{s-2K_{a}-4}(\R^{3}))$ with the existence time satisfying
\ba\label{exist-time-fo}
\liminf_{\e \to 0} T^{*}_{\e} \geq T^{*}.
\ea
Moreover, the WKB solution $U_{a}$ given in Theorem \ref{thm:app-defo} is stable in the following sense:
\ba\label{stable-fo}
\|(U-U_{a})(t,\cdot)\|_{H^{s-2K_{a}-4}(\R^{3})} \leq D(t)\e^{K_{a}+1}, \quad  \mbox{for all $t <\min\{T^{*}, T_{\e}^{*}\}.$}
\ea

\end{itemize}
\end{theorem}

\subsection{Some key observations}\label{sec-keys}

Before proving our main theorems, we give some key observations on the WKB construction which play a crucial role in showing the linear in time growth rate in the error estimates of order $\e^{2}$.

\begin{itemize}

\item[(1)] The Klein-Gordon operator satisfies the so called \emph{weak transparency conditions} following the terminology by Joly-M\'etivier-Rauch \cite{JMR2}, which are exactly the conditions \eqref{wktr-1} in Lemma \ref{lem-wktr} in our setting. Such weak transparency conditions ensure the WKB expansion can be proceeded up to an arbitrary order. Weak transparency conditions arise typically in the WKB expansion of the second most singular order.  In our setting, the most singular order is $O(\e^{-2})$ and the weak transparency conditions appear at order $O(\e^{-1})$ in the WKB construction.  We shall point out that in the present study the nonlinearity is of order $O(1)$ and is not involved in the weak transparency conditions. Essentially, the WKB construction here can be done up to an arbitrary order as long as the nonlinearity $f(u)$ is sufficiently smooth in $u$.  

\item[(2)] The amplitudes $U_{n,p}$ in the WKB expansion are determined by the solutions $g_{n}$ to the (nonlinear) Schr\"odinger equations,  see Propositions \ref{prop-Un1} and \ref{gn=0-n-odd}. In particular, the leading nonlinear  Schr\"odinger equation is exactly the cubic Schr\"odinger equation \eqref{eq-g-cub}, while the others are actually linear Schr\"odinger equations. As a result, the life span of the WKB solution is determined by the life span of the solution to the cubic Schr\"odinger equation \eqref{eq-g-cub}. Hence, in the defocusing case $\l>0$, a global-in-time WKB solution can be constructed.

\item[(3)] Not only the WKB expansion can be proceeded up to an arbitrary order, we also arrived at an initial perturbation $(U-U_{a})(0,\cdot)$ to be small of an arbitrary order, that is $\e^{K_{a}+1}$, which is necessary to obtain the higher order error estimates. This is in general not true for ill prepared data which is our case with $\phi$ and $\psi$ independent of $\e$. Here, imposing $g_{n}(0,\cdot) = 0$ for odd $n$ and with a careful choice of the initial data for $g_{n}$ with $n$ even,  we are able to erase the initial values of higher order harmonics generated by the nonlinearity. Moreover, we observed induction relations for $U_{n,p}$ that ensure that once the initial values for $U_{n}(0, \cdot) = 0$ for some even $n\geq 2$, then automatically $U_{n+1}(0,\cdot) = 0$ holds. See Proposition \ref{ini-Un=0}.

\item[(4)] With an approximate solution $U_{a}$ satisfying \eqref{eq-WKB-fo}, it is rather classical to obtain an error estimate of the form $D(t) \e^{K_{a}+1}$ with $D(t) $ an increasing function in $t$. While in general, $D(t)$  either blows up as $t\to T^{*}$ in the focusing case or grows exponentially in $t$ for the defocusing case where the approximate solution exists globally. To have linear growth rate $(1+t)$ for $D(t)$, it is crucial not only to use the decay in time estimates of $g_{0}$ which is the solution to the cubic Schr\"odinger equation, but also to derive the specific growth rates for the higher order perturbations $g_{n}, n\geq 1$ which are the solutions to linear Schr\"odinger equations. See Proposition \ref{prop-reg-gn-defo}.   

\item[(5)] Similar results (global-in-time WKB solutions, long time approximation, convergence rates estimates) hold for the linear case $\l = 0$ as the defocusing case $\l >0$. The proofs follow step by step those of the defocusing case.

\end{itemize}

We further give a remark concerning the choice of nonlinearities. 
\begin{remark}\label{rem-nonlinear}
In this paper we used essentially four properties coming from the cubic structure $f(u) = \lambda u^{3}$:
\begin{itemize} 
\item[(i).] The smoothness of $f(u)$ in $u$. This ensures 
$$\|f^{(n)}(u)\|_{H^{s}} \leq C (\|u\|_{H^{s}})\|u\|_{H^{s}} , \quad s>\frac{3}{2}$$
 where $f^{(n)}$ is the $n$-th derivative of $f$. 
\item[(ii).] In the defocusing case, the solution $g_{0}$ to the cubic NLS enjoys decay-time estimates: 
$$\|g_{0}(t)\|_{L^{\infty}} \leq C (1+t)^{-\frac{3}{2}}.$$
\item[(iii).] The linearized operator $f'(g_{0}) \sim g_{0}^{2}$ decays like $t^{-\frac{3}{2}}$, that is 
\ba\| f'(g_{0}) u\|_{H^{s}}  \leq C \big(\| f'(g_{0})\|_{L^{\infty}} \|u\|_{H^{s}} + \| f'(g_{0})\|_{H^{s}} \|u\|_{L^{\infty}} \big) \leq C (1+t) \|u\|_{H^{s}}, \quad s>\frac 32.
\nn\ea
\item[(iv).] The structure allows us to have a WKB solution coinciding initially to the exact solution up to an arbitrary high order. 
\end{itemize}
As long as the above three properties are satisfied for certain nonlinearity, similar results can be shown following the arguments in the present paper.

\end{remark}

We finally give a remark which gives some connections between the present study and the diffractive optics. 
\begin{remark}\label{rem-dif}
In nonlinear geometric optics, weak transparency conditions imply that the usual geometric optics scheme applied to a nonlinear equation leads to a linear transport equation with speed of propagation $\nabla\omega(k)$ where $k$ is the spatial wave number and $(\omega, k)$ lies in the characteristic of the hyperbolic operator. One way to reach the nonlinear regime is to work on long time interval and introduce a slow time variable $\tau = \e t$ and one expects a nonlinear Schr\"odinger equation to describe the amplitude of the leading term in the WKB solution, where the Schr\"odinger operator takes the form $$  \frac{i}{2} \nabla^{2} w(k) (\nabla_{x} , \nabla_{x}) =   \frac{i}{2}  \sum_{p,q=1}^{3} \d_{k_{p} k_{q}} \omega(k) \d_{x_{p}x_{q}}. $$
We refer to the classical paper \cite{DJMR96} and review paper \cite{Dumas1}.  In the present study, the initial spatial wave number is chosen $k = 0$ and the Klein-Gordon characteristic is given by $\o(\xi) =  (1+ |\xi|^{2})^{\frac 12}$ which is smooth in $\xi\in \R^{3}$. Consequently, the speed of propagation in the linear transport equation is $(\nabla_{\xi} \omega)(0) $ which is zero. This actually corresponds to the weak transparency condition in Lemma \ref{lem-wktr}. Moreover, the Schr\"odinger operator becomes
$$
 \frac{i}{2}  \sum_{i,j=1} (\d_{\xi_{i} \xi_{j}} \omega)(0) \d_{x_{i}x_{j}} =  \frac{i}{2}\Delta_{x} .
$$
This corresponds to the algebraic Lemma \ref{lem-2ndAlg}.

\end{remark}


\section{WKB cascade and approximate solutions}\label{sec:wkb}
We carry out the ideas of Section \ref{sec-keys} to construct approximate solutions by WKB expansion. We start with orders $\e^{-2}$ to $\e^{2}$, and then proceed with the general case $\e^{n}$ by induction.

\subsection{Some notations and useful lemmas in geometric optics}
Let $p\in \Z$ and introduce $$L_{p} : = (ipI_{d+2}+A_0),\quad \mbox{for each $p\in \Z$}.$$
Let $\Pi_{p}$ be the orthogonal projection onto $\ker L_{p}$:
\be\label{wkb--2-Pi-2}
\Pi_{p} L_{p} = L_{p}\Pi_{p} = 0.
\nn\ee

If $L_{p}$ is invertible, let $L^{-1}_{p}$ be the inverse of $L_{p}$. If $L_{p}$ is not invertible,  let $L_{p}^{-1}$ be the partial inverse of $L_{p}$ such that
\be\label{wkb--2-Lp-2}
\Pi_{p}L^{-1}_{p} = L^{-1}_{p} \Pi_{p} = 0, \quad  L^{-1}_{p} L_{p} = L_{p}L_{p}^{-1} = I_{d+2}-\Pi_{p}.\nn
\ee

 It is easy to find that the matrices $L_{p}= (ipI_{d+2}+A_0)$ are invertible except $p\in {\cal H}_0 = \{-1,0,1\}$. Indeed, for $|p| \geq 2$, we have
 \be\label{Lp-p>2}
 L_{p}^{-1} = \bp-\frac{i}{p}  I_{d} & 0_d & 0_d \\  0_d^{\rm T} & \frac{ip}{1-p^{2}} & \frac{-1}{1-p^{2}} \\ 0_d^{\rm T} & \frac{1}{1-p^{2}} & \frac{ip}{1-p^{2}}\ep, \quad |p|\geq 2.
 \ee

 For $p\in {\cal H}_{0}$,  by direct computation we have
\be\label{intr-pola-0}
\ker A_{0} = (w, 0,0) \ \mbox{for any $w\in \R^{3}$},\quad \ker ( \pm i I_{d+2}+A_0) = {\rm span\,}\{e_{\pm}\},
\ee
where the conjugate couple
$
e_\pm:= (0_d^{\rm T},\pm i,1)^{\rm T}.
$
Hence, the orthogonal projections $\Pi_{\pm1}$ are defined as
\be\label{intr-pola-1}
\Pi_{\pm1} V = \frac{1}{|e_{\pm}|^{2}} (V, e_{\pm}) e_{\pm} = \frac{1}{2}(V, e_{\pm}) e_{\pm}, \quad \forall\, V \in \R^{d+2},
\ee
where $(U,V)$ denotes for the usual inner product in $\C^{d+2}$.  

Moreover, by direct computation, we have
\be\label{Pi-pm0}
\Pi_{0} = \bp I_{d} & 0_d & 0_d \\ 0_d^{\rm T}& 0&0 \\ 0_d^{\rm T}&0&0\ep, \quad \Pi_{1}= \frac{1}{2}\bp 0_{d\times d} & 0_d & 0_d \\ 0_d^{\rm T}& 1 & i \\ 0_d^{\rm T}&-i&1\ep, \quad \Pi_{-1} = \overline\Pi_{1},
\ee
and
\be\label{Lp-p=1}
L_{0}^{-1} = \bp 0_{d\times d} & 0_d & 0_d \\ 0_d^{\rm T}& 0 & -1 \\ 0_d^{\rm T} & 1 & 0 \ep, \quad L_{1}^{-1} = \bp -i\, I_{d} & 0_d & 0_d \\ 0_d^{\rm T}& -\frac{i}{4}&-\frac{1}{4} \\ 0_d^{\rm T}&\frac{1}{4}&-\frac{i}{4}\ep, \quad L^{-1}_{-1} = \overline {L_{1}^{-1}}.
\ee

\medskip

We now introduce three lemmas. The first one corresponds to certain compatibility conditions:
\begin{lemma}\label{lem-wktr} There holds
\be\label{wktr-1}
\Pi_{p} A(\nabla) \Pi_{p}   = 0, \quad \mbox{for each $p\in \Z$}.
\ee

\end{lemma}

The results in Lemma \ref{lem-wktr} can be shown by direct computation using \eqref{intr-pola-1} and \eqref{Pi-pm0}. We will omit the details.

\medskip

The second one corresponds to some algebraic lemma:
\begin{lemma}\label{lem-2ndAlg} There holds
\be\label{2nd-Alg}
 \Pi_{1} A(\nabla)  L^{-1}_{1} A(\nabla) \Pi_{1} = - \frac{i}{2}\Delta \Pi_{1}.
\ee
\end{lemma}
This result is actually a consequence of the second algebraic lemma in nonlinear geometric optics, see for example \cite{Texier04}. While here we already have the precise expressions of $\Pi_{1}$ and $L^{-1}_{1}$, the equality in \eqref{2nd-Alg} can be shown by direct computations.

\medskip

The last one is the following of which the proof is rather straightforward:
\begin{lemma}\label{lem-1-pi1} Let $g$ and $h$ be scalar functions. Then there hold
 \ba\label{wkb-U01-03}
 L^{-1}_{1} A(\nabla)  (g e_{+} )= \bp \nabla g \\  0 \\ 0 \ep, \  L^{-1}_{1} A(\nabla)  L^{-1}_{1} A(\nabla) (g e_{+}) + L^{-1} \bp 0_d \\  h\\ 0 \ep  = \frac{1}{4}  \bp 0_d \\ -  i (\Delta g + h ) \\ \Delta g + h \ep.
\ea
\end{lemma}

\subsection{Order $O(\e^{-2})$} \label{sec-wkb-e--2}
 As shown in Section \ref{sec:intr-wkb}, by plugging the formal expansion \eqref{def-wkb} into \eqref{00},  the equations which comprise all terms of order $O(\e^{-2})$ are exactly $\Phi_{-2,p}=0$:
\be\label{wkb--2}
L_{p} U_{0,p} = 0 \Longleftrightarrow \Pi_{p} U_{0,p}=U_{0,p},\quad \mbox{for each $p\in \Z$}.\ee
This implies
\be\label{intr-pola}
 U_{0,1}= \Pi_{1} U_{0,1} =  g_{0} e_{+} \quad \mbox{for some scalar function $ g_0$.}
\ee

We then deduce from \eqref{wkb--2} and the invertibility of $L_{p}$ that
\be\label{U0p-p>2}
U_{0,p}=0,\quad \mbox {for all $p$ such that $|p|\geq 2$}.
\ee

For $p=-1$,  to guarantee the reality of the solutions, it is natural and necessary to impose
\be\label{U0-1}
U_{0,-1}=\overline U_{0,1}=\bar g_0 e_-.
\ee

\begin{remark}\label{rem-neg-p}
Similar requirement  as \eqref{U0-1} is needed for other terms to ensure the reality of the WKB solutions:
\be\label{web-0-p<0}
U_{n,p} = \overline U_{n,-p}, \quad \mbox{for each $n,p \in \Z$}.\nn
\ee
As a result, in the sequel we shall only consider nonnegative $p$.

\end{remark}

We find it is not needed to include the mean mode $U_{0,0}$. Hence, for simplicity, we take
$
U_{0,0}=0.
$
To sum up, the leading term $U_{0}$ has the form
\be\label{U0-form}
U_{0} = e^{i\th} g_{0} e_{+} + e^{-i\th} \bar g_{0} e_{-},
\ee
where $g_{0}$ is a scalar function to be made precise later.

Before continuing the WKB cascade, we find that, with the cubic nonlinearity we have chosen  and with such a leading term $U_{0}$ as in \eqref{U0-form}, it is sufficient to consider WKB solutions only involving odd harmonics (odd $p\in \Z$):
\begin{proposition}\label{prop-p-odd}
Let $U_{a}$ be defined as in \eqref{def-wkb} and let $U_{0}$ be given as in \eqref{U0-form}. If
\be\label{0-p-even}
U_{n,p} = 0, \ \mbox{for all $p\in   \Z_{\rm even}$,  for each $n = 0, \cdots, K_{a}+2$},
\ee
 then the equations $\Phi_{n,p} = 0$ are satisfied for each $n = -2,-1, 0 , \cdots, K_{a} $ and all $p\in  \Z_{\rm even}$.

\end{proposition}

\begin{proof}[Proof of Proposition \ref{prop-p-odd}] With the assumption \eqref{0-p-even}, there holds for each $n =-2, -1,  0, \cdots, K_{a}$ and for each $p\in  \Z_{\rm even}$ that
\be\label{0-p-even-1}
\d_t U_{n,p}- A(\nabla) U_{n+1,p} +\big(ip\, I_{d+2} +A_0\big) U_{n+2,p}  = 0.
\ee
Clearly $p_{1} + p_{2} + p_{3} $ is odd when $p_{1}, p_{2}, p_{3}$ are all odd. Hence, by the form of $F(U_{a})_{n,p}$ in \eqref{Fnp}, we have $f(u_{a})_{n,p} = 0$ and then $F(U_{a})_{n,p} = 0$ for each even $p\in \Z_{\rm even}$.  Together with \eqref{0-p-even-1}, we have $\Phi_{n,p} = 0$  for each even $p\in \Z_{\rm even}$, for each $n = -2,-1, 0, \cdots, K_{a}$.
\end{proof}

\begin{remark}\label{rem-p-odd} With such an observation as in Proposition \ref{prop-p-odd},  in the sequel we will impose $U_{n,p} = 0$ for all $n\in \N, \ p\in \Z_{\rm even}$ and it is sufficient to consider \eqref{phi-n-p} with $p$ odd.
 
\end{remark}

\subsection{Order $O(\e^{-1})$}\label{sec-wkb-e--1}
 We then consider the equations which comprise all terms of order $O(\e^{-1})$, that are $\Phi_{-1,p}=0$:
\be\label{wkb--1}
-A(\nabla) U_{0,p}+ L_{p} U_{1,p}=0, \quad \mbox{for each $p\in \Z$}.
\ee
As observed from Remark \ref {rem-neg-p}, Proposition \ref{prop-p-odd} and Remark \ref{rem-p-odd}, it is sufficient to consider \eqref{wkb--1} with odd $p\geq 0$.

When $p=1$, applying the partial inverse $L^{-1}_{1}$ to \eqref{wkb--1} implies
\be\label{U11-1}
(1-\Pi_{1}) U_{1,1} = L^{-1}_{1} A(\nabla) U_{0,1} = L^{-1}_{1} A(\nabla) \Pi_{1} U_{0,1}  = L^{-1}_{1} \bp i\nabla g_{0}\\ 0\\0\ep =  \bp \nabla g_{0}\\ 0\\0\ep,
\ee
where we used \eqref{intr-pola} and \eqref{Lp-p=1}. While for the other part $\Pi_{1} U_{1,1}$, by \eqref{intr-pola-0}--\eqref{intr-pola-1}, we have
\be\label{U11-2}
\Pi_{1} U_{1,1} = g_{1} e_+,  \quad \mbox{for some scalar function $g_1$}.\nn
\ee

\smallskip

When $p\geq 3$, by the invertibility of $L_{p} = (ipI_{d+2}+A_0)$ and \eqref{U0p-p>2}, we have
\be\label{U1p-p>2}
U_{1,p}=0,\quad \mbox{for all $p$ such that $p\geq 3$}.\nn
\ee

\medskip

Thus, $U_{1}$ has the form
\be\label{U1-form}
U_{1} = e^{i\th} \left[g_{1} e_{+} +\bp \nabla g_{0}\\ 0\\0\ep \right] + c.c.,  \quad \mbox{for some scalar function $g_1$}
 \ee
where $c.c.$ denotes the related complex conjugate.

\subsection{Order $O(\e^{0})$} \label{sec-wkb-e-0} The equations comprising all the terms of order $O(\e^{0})$ are $\Phi_{0,p}=0$:
\be\label{wkb-0}
\d_t U_{0,p}-A(\nabla) U_{1,p} + L_{p} U_{2,p}=F(U_{a})_{0,p}, \quad \mbox{for each $p\in \Z$}.
\ee
 By the form of the nonlinear term $F$ in \eqref{def:AB} and the form of the leading term $U_{0}$ obtained in \eqref{U0-form} in the previous step, we have
$ F(U_{a})_{0}=(0_d^{\rm T},- f(u_{0}),0)^{\rm T}$ with
\ba\label{f1-th}
f(u_0) = \l ( e^{i\th} g_0 + e^{-i\th}\bar g_0)^{3} =  \l \big( e^{3 i\th} g_0^{3} + 3 e^{i\th} |g_0|^{2} g_{0} +  3 e^{- i\th} |g_0|^{2} \bar g_{0} +  e^{-3 i\th} \bar g_0^{3} \big),
\ea
where we used the fact that $f(u_{a})_{0} = f(u_{0})$.

\medskip

When $p=1$, equation \eqref{wkb-0} becomes
\be\label{wkb-U01-0}
\d_t U_{0,1}- A(\nabla) U_{1,1} + L_{1} U_{2,1} = F(U_{a})_{0,1} = \bp 0_d \\  -3 \l |g_{0}|^{2} g_{0} \\ 0 \ep.
\ee
We first apply $\Pi_{1}$ to \eqref{wkb-U01-0} and derive
\be\label{wkb-U01-1}
\d_t \Pi_{1} U_{0,1} - \Pi_{1} A(\nabla) U_{1,1}  = \Pi_{1} \bp 0_d \\  -3 \l |g_{0}|^{2} g_{0} \\ 0 \ep.
\ee
Using \eqref{U11-1} allows us to do the decomposition
\ba\label{wkb-U01-2}
\Pi_{1} A(\nabla) U_{1,1}  & = \Pi_{1} A(\nabla) (1-\Pi_{1})U_{1,1} +  \Pi_{1} A(\nabla) \Pi_{1} U_{1,1} \\
& =  \Pi_{1} A(\nabla)  L^{-1}_{1} A(\nabla) \Pi_{1} U_{0,1} +  \Pi_{1} A(\nabla) \Pi_{1} U_{1,1} .
\ea
 
We then deduce from \eqref{wkb-U01-1}, \eqref{wkb-U01-2}, Lemmas \ref{lem-wktr}  and \ref{lem-2ndAlg} that
\be\label{wkb-U01-3}
\d_t \Pi_{1} U_{0,1} + \frac{i}{2}\Delta \Pi_{1} U_{0,1} = \Pi_{1} \bp 0_d \\  -3 \l |g_{0}|^{2} g_{0} \\ 0 \ep.
\ee
Together with \eqref{intr-pola}, we deduce from \eqref{wkb-U01-3} that
\be\label{wkb-U01-4}
2 i \d_t g_{0}  - \Delta  g_{0}  + 3 \l |g_{0}|^{2} g_{0} = 0,
\ee
which is exactly the cubic Schr\"odinger equation in \eqref{eq-g-cub}.

Recall the initial data of $U(0,\cdot)$ in \eqref{ini-data00}:
\be\label{ini-data000}
U(0,\cdot) =  (0_d^{\rm T} , \psi ,\phi   )^{\rm T} +  (\e \nabla^{\rm T} \phi , 0,0)^{\rm T}.
\ee
We will choose the initial datum of $g_{0}$ such that
$
U_{0} (0,\cdot) = (0_d^{\rm T} , \psi ,\phi )^{\rm T}.
$
By \eqref{U0-form}, it is sufficient to impose
$$
g_{0}(0,\cdot) e_{+} + \bar g_{0}(0,\cdot) e_{-} =  (0_d^{\rm T} , \psi ,\phi )^{\rm T}.
$$
This forces
\be\label{wkb-ini-g0}
g_0(0,\cdot)=\frac{\phi - i \psi}{2},
\ee
which is exactly the initial datum in \eqref{eq-g-cub}.  With such a choice, we see from  \eqref{intr-pola},  \eqref{U11-1} and \eqref{ini-data000} that
\be\label{wkb-ini-o1}
U_{0} (0,\cdot)  +\e (1-\Pi_{1}) U_{1,1} (0,\cdot) +\e (1-\Pi_{-1}) U_{1,-1} (0,\cdot)  = U(0,\cdot).
\ee
Thus in the next order it is nature to require $\Pi_{\pm 1} U_{1, \pm 1}(0, \cdot) = 0$.

\medskip

We then apply the partial inverse $L_{1}^{-1}$ to \eqref{wkb-U01-0} and use \eqref{wkb-U01-4} to derive
\ba\label{wkb-U01-01}
 (1-\Pi_{1}) U_{2,1} &=  L^{-1}_{1} A(\nabla)  U_{1,1}  + L^{-1}_{1} \bp 0_d \\  -3 \l |g_{0}|^{2} g_{0} \\ 0 \ep \\
 &=  L^{-1}_{1} A(\nabla)  \Pi_{1} U_{1,1} + L^{-1}_{1} A(\nabla)  L^{-1}_{1} A(\nabla) \Pi_{1} U_{0,1} + L^{-1}_{1} \bp 0_d \\  -3 \l |g_{0}|^{2} g_{0} \\ 0 \ep.
\ea
Then by \eqref{wkb-U01-01}, \eqref{wkb-U01-03} and \eqref{wkb-U01-4}, we deduce that
 \ba\label{wkb-U01-04}
  (1-\Pi_{1}) U_{2,1}   = \bp \nabla g_{1} \\  0 \\ 0 \ep + \frac{1}{4} \bp 0_d \\ -  i (\Delta  g_{0} -  3 \l |g_{0}|^{2} g_{0}) \\ \Delta  g_{0} -  3 \l |g_{0}|^{2} g_{0} \ep  =  \bp \nabla g_{1} \\  0 \\ 0 \ep + \frac{1}{2} \bp 0_d \\  \d_{t}g_{0}\\ i \d_{t}g_{0} \ep.
  \ea

While for the other part $\Pi_{1} U_{2,1}$, there exists some scalar function $g_{2}$ such that
\be\label{wkb-U01-05}
\Pi_{1} U_{2,1} = \big(g_{2} - \frac{i}{2} \d_{t} g_{0}\big) e_{+}.
\ee
Then we see from \eqref{wkb-U01-04} and \eqref{wkb-U01-05} that
\be\label{wkb-U01-06}
 U_{2,1} =   (1-\Pi_{1}) U_{2,1}  +   \Pi_{1} U_{2,1}  =  g_{2}  e_{+} + \bp \nabla g_{1} \\  \d_{t}g_{0} \\ 0 \ep , \quad \mbox{for some scalar function $g_{2}$}.
\ee

\medskip

While for $p=3$, by \eqref{Lp-p>2} and \eqref{f1-th}, together with the fact $U_{0,3} = U_{1,3} = 0$, equation \eqref{wkb-0} becomes
\be\label{web-0-p=3}
L_{3} U_{2,3} = \bp 0_d \\ -\l g_{0}^{3} \\ 0 \ep \Longleftrightarrow U_{2,3} = L_{3}^{-1} \bp 0_d \\ -\l g_{0}^{3} \\ 0 \ep = \frac{\l g_{0}^{3}}{8}\bp 0_d \\ 3i \\ 1\ep.
\nn
\ee

While for $p\geq 5$, it is straightforward to deduce
\be\label{web-0-p>4}
L_{p} U_{2,p} = 0 \Longleftrightarrow U_{2,p} = 0, \quad \forall\, p\geq 5.
\nn
\ee

Hence $U_{2}$ has the form
\be\label{U2-form}
U_{2} = e^{i\th} \left[g_{2} e_{+} +\bp \nabla g_{1} \\  \d_{t}g_{0} \\ 0 \ep \right] +  e^{3i\th} \frac{\l g_{0}^{3}}{8}\bp 0_d \\ 3i \\ 1\ep + c.c..
\ee

\subsection{Order $O(\e^{1})$ } \label{sec-wkb-e-1} The equations comprise all the terms of order $O(\e^{1})$ are $\Phi_{1,p}=0$:
\be\label{wkb-1}
\d_t U_{1,p}-A(\nabla) U_{2,p} + L_{p} U_{3,p} = F(U_{a})_{1,p}, \quad \mbox{for each $p\in \Z$}.
\ee
By \eqref{U0-form} and \eqref{U1-form}, we know the right-hand side $ F(U_{a})_{1,p} = (0_d^{\rm T}, -  f(u_{a})_{1,p},0)^{\rm T}$ with
\ba\label{f1}
f(u_{a})_{1} &= 3 \l u_{0}^{2} u_{1} = 3\l (e^{i\th} g_{0} + e^{-i\th} \bar g_{0})^{2} (e^{i\th} g_{1} + e^{-i\th} \bar g_{1}) \\
& = 3\l \big (e^{3 i\th} g_{0}^{2} g_{1} + e^{i\th} (g_{0}^{2} \bar g_{1}   + 2 |g_{0}|^{2} g_{1})  + e^{-i\th} (\bar g_{0}^{2}   g_{1}   + 2 |g_{0}|^{2} \bar g_{1}) +e^{- 3 i\th} \bar g_{0}^{2} \bar g_{1}  \big) .
\nn
\ea

\medskip

For $p=1$, equation \eqref{wkb-1} becomes
\be\label{wkb-11}
\d_t U_{1,1}-A(\nabla) U_{2,1} + L_{1} U_{3,1} =  \bp 0_d \\ -3 \l  (g_{0}^{2} \bar g_{1}   + 2 |g_{0}|^{2} g_{1}) \\ 0 \ep.
\ee

Applying $\Pi_{1}$ to \eqref{wkb-11},  using Lemma \ref{lem-wktr}, together with \eqref{wkb-U01-01}--\eqref{wkb-U01-04}, gives
\be\label{wkb-11-1}
\d_t \Pi_{1} U_{1,1} - \Pi_{1} A(\nabla)  \bp \nabla g_{1} \\  0 \\ 0 \ep  - \frac{1}{2}  \Pi_{1} A(\nabla) \bp 0_d \\  \d_{t}g_{0} \\ i \d_{t}g_{0}  \ep =  \Pi_{1} \bp 0_d \\ -3 \l  (g_{0}^{2} \bar g_{1}   + 2 |g_{0}|^{2} g_{1}) \\ 0 \ep.
\ee
Observe that
 \ba\label{PiA-Pi12}
 \Pi_{1} A(\nabla)  \bp \nabla g \\  0 \\ 0 \ep  = \Pi_{1} \bp 0_d \\  \Delta  g \\ 0 \ep, \quad \Pi_{1} A(\nabla) \bp 0_d \\  g \\ h  \ep = 0, \quad \mbox{for any scalar functions $g$ and $h$.}
 \ea
Then equation \eqref{wkb-11-1} is equivalent to
\be\label{wkb-11-2}
2 i \d_t g_{1} - \Delta  g_{1}  + 3 \l  (g_{0}^{2} \bar g_{1}   + 2 |g_{0}|^{2} g_{1})  = 0,
\ee
which is a linear Schr\"odinger equation in $g_{1}$.

\medskip

 Applying $L_{1}^{-1}$ to \eqref{wkb-11}, using  \eqref{wkb-U01-04}--\eqref{wkb-U01-06} and Lemma \ref{lem-1-pi1} gives
\ba\label{wkb-11-3}
(1-\Pi_{1}) U_{3,1}     = -\d_t L_{1}^{-1} U_{1,1}  + L_{1}^{-1} A(\nabla) U_{2,1}  + L_{1}^{-1} \bp 0_d \\ - 3 \l  (g_{0}^{2} \bar g_{1}   + 2 |g_{0}|^{2} g_{1}) \\ 0 \ep     =   \bp \nabla g_{2} \\  \frac{\d_{t}g_{1}}{2} \\ \frac{i \d_{t}g_{1}}{2}  \ep.
\ea
While for $\Pi_{1} U_{3,1},$ there exists a scalar function $g_{3}$ such that
 \be\label{wkb-11-3-PiU31}
\Pi_{1} U_{3,1} = \big(g_{3} - \frac{i}{2} \d_{t} g_{1}\big) e_{+}.\nn
\ee
Then
\be\label{wkb-11-3-U31}
 U_{3,1} =(1-\Pi_{1}) U_{3,1}   + \Pi_{1} U_{3,1}   = g_{3}  e_{+} + \bp \nabla g_{2} \\  \d_{t}g_{1} \\ 0 \ep , \quad \mbox{for some scalar function $g_{3}$}.\nn
\ee

\medskip

For $p=3$, equation \eqref{wkb-1} becomes
 \be\label{wkb-13-1}
-A(\nabla) U_{2,3} + L_{3} U_{3,3} = F(U_{a})_{1,3},\nn
\ee
 which implies
\ba\label{wkb-13-2}
U_{3,3}  = L_{3}^{-1} A(\nabla) U_{2,3} + L_{3}^{-1}  F(U_{a})_{1,3}  =  \frac{\l }{8}\bp \nabla(g_{0}^{3}) \\ 0 \\ 0 \ep + \frac{3\l g_{0}^{2} g_{1}}{8}\bp 0_d \\ 3 i \\ 1 \ep.
\nn
\ea

Finally, it is easy to obtain
\ba\label{wkb-14}
U_{3,p}  = 0, \ \forall p\geq 5.
\ea

From \eqref{wkb-11-3}--\eqref{wkb-14} we deduce
\be\label{U3-form}
U_{3} = e^{i\th} \left[g_{3} e_{+} +\bp \nabla g_{2} \\ \d_{t} g_{1} \\0\ep  \right] +  e^{3i\th}  \left[\frac{\l }{8}\bp \nabla(g_{0}^{3}) \\ 0 \\ 0 \ep + \frac{3\l g_{0}^{2} g_{1}}{8}\bp 0_d \\ 3 i \\ 1 \ep \right] + c.c..
\ee

Thanks to the observation \eqref{wkb-ini-o1}, we do not need to include any initial datum of  $g_{1}$ to ensure an initial difference arbitrarily small between the exact solution $U$ and the WKB approximate solution $U_{a}$. Thus, we choose
\be\label{ini-g1}
g_{1}(0,\cdot) = 0.
\ee
Then the solution of \eqref{wkb-11-2} is identical zero, i.e.
$
g_{1}\equiv 0.
$

\subsection{Order $O(\e^{2})$}\label{sec-wkb-e-2} The equations comprise all the terms of order $O(\e^{2})$ are $\Phi_{2,p}=0$:
\be\label{wkb-2}
\d_t U_{2,p}-A(\nabla) U_{3,p} + L_{p} U_{4,p} = F(U_{a})_{2,p}, \quad \mbox{for each $p\in \Z$},
\ee
where the right-hand side $F(U_{a})_{2}= (0_d^{\rm T}, - f(u_{a})_{2},0)^{\rm T}$ with
$
f(u_{a})_{2} = 3 \l ( u_{0} u_{1}^{2} + u_{0}^{2} u_{2}).
$
By \eqref{U1-form} and $g_{1} = 0$, we know that $u_{1} = 0$. Together with \eqref{U0-form} and \eqref{U2-form} we have
\ba\label{def-f2-1}
f(u_{a})_{2} & = 3 \l   u_{0}^{2} u_{2}= 3\l \big(e^{i\th} g_{0} +e^{-i\th} \bar g_{0} \big)^{2} \left[  e^{i\th} g_{2}  + e^{-i\th} \bar g_{2}   + e^{3i\th} \frac{\l}{8} g_{0}^{3} + e^{-3i\th} \frac{\l}{8} \bar g_{0}^{3} \right]\\
 & = 3 \l \left[ e^{5i\th} \frac{\l}{8}g_{0}^{5} + e^{3 i\th}\big( g_{0}^{2} g_{2} + \frac{\l}{4} |g_{0}|^{2} g_{0}^{3} \big)   + e^{i\th} \big(  g_{0}^{2} \bar g_{2} + 2 |g_{0}|^{2} g_{2} +  \frac{\l}{8} |g_{0}|^{4} g_{0} \big) \right] + c.c.
\ea

\medskip

For $p=1$, equation \eqref{wkb-2} becomes
\be\label{wkb-21}
\d_t U_{2,1}-A(\nabla) U_{3,1} + L_{1} U_{4,1} =  \bp 0_d \\  -  f(u_{a})_{2,1}\\ 0 \ep,
\ee
where, by \eqref{def-f2-1},  $f(u_{a})_{2,1}$ has the form
\ba\label{wkb-21-0}
f(u_{a})_{2,1} : = 3\l \big(g_{0}^{2} \bar g_{2} + 2 |g_{0}|^{2} g_{2} +  \frac{\l}{8} |g_{0}|^{4} g_{0}  \big).
\ea

Applying $\Pi_{1}$ to \eqref{wkb-21},  using Lemma \ref{lem-wktr}, together with \eqref{PiA-Pi12} and \eqref{wkb-11-3}, gives
\be\label{wkb-21-1}
\d_t \Pi_{1} U_{2,1}  - \Pi_{1} \bp 0_d \\  \Delta  g_{2} \\ 0 \ep =  \Pi_{1} \bp 0_d \\ - f(u_a)_{2,1} \\ 0 \ep,
\ee
which is equivalent to, by using \eqref{wkb-U01-05},
\be\label{wkb-21-2}
2 i \d_t   \big(g_{2} - \frac{i}{2} \d_{t} g_{0}\big) - \Delta  g_{2}  + f(u_{a})_{2,1} = 0 \ \iff \  2 i \d_t   g_{2} +  \d_{tt} g_{0} - \Delta  g_{2}  + f(u_{a})_{2,1} = 0.
\ee
Here we shall impose a proper initial datum for $g_{2}$ in order to have as small as possible initial difference between $U$ and $U_{a}$. With the observation in \eqref{wkb-ini-o1} and \eqref{ini-g1}, there holds
\be\label{ini-diff-1}
U(0,\cdot) - U_{0}(0,\cdot) - \e U_{1}(0,\cdot) = 0.\nn
\ee
In order to obtain a possibly smallest initial difference between the exact solution and the WKB approximate solution, we shall choose the initial datum of $g_{2}$  such that
\be\label{g20-0}
U_{2}(0,\cdot) = 0.
\ee
By \eqref{U2-form} and the fact $g_{1}(0,\cdot) = 0$, in order to achieve \eqref{g20-0}, it suffices to impose
\be\label{g20-1}
\left\{
\begin{aligned} & i g_{2}(0,\cdot) + \d_{t} g_{0}(0,\cdot)  + \frac{3 i \l}{8} g_{0}^{3}(0,\cdot) + c.c. = 0, \\
& g_{2}(0,\cdot)  + \frac{\l}{8} g_{0}^{3}(0,\cdot) + c.c.  = 0.
\end{aligned}
\right.
\nn
\ee
This implies that
\ba\label{g20-2}
Im(g_{2}(0,\cdot)) & = Re\left[ \d_{t} g_{0}(0,\cdot)  + \frac{3i \l}{8} g_{0}^{3}(0,\cdot) \right] =  -\frac{1}{4} \Delta  \psi + \frac{21\l}{64} \phi^{2} \psi + \frac{9\l}{64} \psi^{3},\\
Re(g_{2}(0,\cdot)) & = - Re \left[  \frac{\l}{8} g_{0}^{3}(0,\cdot) \right]   =   -\frac{\l}{64} \big( \phi^{3} -3 \phi \psi^{2}\big),
\ea
where $Im(g)$ and $Re(g)$ represent the imaginary and real parts of $g$, respectively.

\medskip

By \eqref{U2-form}, \eqref{U3-form}, Lemma \ref{lem-1-pi1}, \eqref{wkb-21-1} and \eqref{wkb-21-2}, applying $L_{1}^{-1}$ to \eqref{wkb-21} gives
\ba\label{wkb-21-3}
 (1-\Pi_{1}) U_{4,1}     = -\d_t L_{1}^{-1} U_{2,1}  + L_{1}^{-1} A(\nabla) U_{3,1}  + L_{1}^{-1} \bp 0_d \\ -  f(u_{a})_{2,1}  \\ 0 \ep     =   \bp \nabla g_{3} \\  \frac{\d_{t}g_{2}}{2} \\ \frac{i \d_{t}g_{2}}{2}  \ep,
\nn
\ea
while for $\Pi_{1} U_{4,1},$ there exists a scalar function $g_{4}$ such that $\Pi_{1} U_{4,1} = \big(g_{4} - \frac{i}{2} \d_{t} g_{2}\big) e_{+}.$

\medskip

For $p=3$, equation \eqref{wkb-2} becomes
\ba\label{wkb-23}
 L_{3}U_{4,3}&=- \d_{t} U_{2,3} +  A(\nabla) U_{3,3} -    \bp 0_d \\ f(u_{a})_{2,3} \\ 0  \ep \\
 &=- \frac{\l \d_{t}(g_{0}^{3})}{8} \bp 0_d \\ 3i \\ 1  \ep+ \frac{\l}{8} \Delta (g_{0}^{3}) \bp 0_d \\ 1 \\ 0  \ep - f(u_{a})_{2,3}  \bp 0_d \\ 1 \\ 0  \ep,
 \nn
\ea
where, by \eqref{def-f2-1},  $f(u_{a})_{2,3}= 3\l \big(g_{0}^{2} g_{2} + \frac{\l}{4} |g_{0}|^{2} g_{0}^{3}\big).$
Thus,
\ba\label{wkb-23-2}
U_{4,3}  = - \frac{\l \d_{t}(g_{0}^{3})}{8} \bp 0_d \\ \frac{5}{4} \\ -\frac{3i}{4}  \ep + \Big[\frac{\l}{8} \Delta (g_{0}^{3})   - 3\l \big( g_{0}^{2} g_{2}  + \frac{\l}{4} |g_{0}|^{2} g_{0}^{3} \big) \Big] \bp 0_d \\ -\frac{3i}{8}\\ -\frac{1}{8}  \ep.
\ea

While for $p=5$, equation \eqref{wkb-2} becomes
\ba\label{wkb-25}
 L_{5}U_{4,5}      =  - \frac{3\l^{2}}{8}g_{0}^{5} \bp 0_d \\ 1 \\ 0  \ep \Longleftrightarrow U_{4,5} =  - \frac{3\l^{2}}{8}g_{0}^{5}  \bp 0_d \\ -\frac{5i}{24}\\ -\frac{1}{24}  \ep.
\ea

For $p\geq 7$, equation \eqref{wkb-2} is equivalent to
\ba\label{wkb-27}
U_{4,p}  = 0, \ \forall \, p\geq 7. \nn
\ea

Thus, $U_{4}$ is of the form:
\be\label{U4-form}
U_{4} = e^{i\th} \left[g_{4} e_{+} + \bp \nabla g_{3} \\ \d_{t} g_{2} \\ 0 \ep  \right] +  e^{3i\th} U_{4,3} +  e^{5i\th} U_{4,5} + c.c., \ \mbox{for some scalar function $g_{4}$,}\nn
\ee
where $U_{4,3}$ and $U_{4,5}$ are given in \eqref{wkb-23-2} and \eqref{wkb-25}, and $g_{2}$ is the solution to \eqref{wkb-21-2}--\eqref{g20-2}.

\subsection{Order $O(\e^{n})$}\label{sec-wkb-e-n}

Based on the WKB expansion up to order $O(\e^{2})$ in previous sections, we will deduce certain induction properties and obtain the complete form of the WKB solutions. 

Firstly, with the observation in Proposition \ref{prop-p-odd} and Remak \ref{rem-p-odd}, we shall impose
\be\label{Unp-p-even}
U_{n,p} = 0, \quad \forall n\in \{0,1,2,\cdots, K_{a}+2\}, \ \forall p\in \Z_{\rm even}.
\ee
 
\subsubsection{$U_{n,\pm 1}$.} For leading harmonics with $p=\pm1$, we have the following conclusion:
\begin{proposition}\label{prop-Un1}
There exist scalar functions $g_{n}, n=0,1,2,\cdots, K_{a}+2$ such that
\be\label{Un1-form}
 (1 - \Pi_{1}) U_{n,1} = \bp \nabla g_{n-1} \\  \frac{1}{2} \d_{t}g_{n-2} \\ \frac{i}{2} \d_{t}g_{n-2} \ep, \quad \Pi_{1} U_{n,1} = \big(g_{n} - \frac{i}{2} \d_{t} g_{n-2}\big) e_{+}, \quad U_{n,1} = g_{n} e_{+} + \bp \nabla g_{n-1} \\ \d_{t} g_{n-2} \\ 0 \ep.
\ee
We shall alway impose $g_{-2} = g_{-1} = 0$. Moreover, for each $n \in\{ 0,1,2,\cdots, K_{a}\}$, $g_{n}$ satisfies the Schr\"odinger equation:
\be\label{Schro-gn}
 2 i \d_t   g_{n} +  \d_{tt} g_{n-2} - \Delta  g_{n}  + f(u_{a})_{n,1} = 0.
\ee

\end{proposition}

\begin{proof} We will prove Proposition \ref{prop-Un1} by induction. In the previous sections, we have shown \eqref{Un1-form} for $n=0,1,2,3,4$ and  \eqref{Schro-gn} for $n=0,1,2$.  Let $k \geq 3$. By induction we suppose \eqref{Un1-form} holds for $n \leq k+1$ and  \eqref{Schro-gn} holds for $n \leq k-1$. Now we shall prove \eqref{Un1-form} for $n = k+2$ and  \eqref{Schro-gn} for $n = k$. With $n=k$ and $p=1$, the equation \eqref{phi-n-p} ($\Phi_{k,1} = 0$) becomes
\be\label{wkb-n-2}
\d_t U_{k,1}-A(\nabla) U_{k+1,1} + L_{1} U_{k+2,1} = F(U_{a})_{k,1}, \quad F(U_{a})_{k,1} = ( 0_d^{\rm T}, - f(u_{a})_{k,1}, 0)^{\rm T}.
\ee

Applying $\Pi_{1}$ to \eqref{wkb-n-2} implies
\be\label{wkb-n-3}
\d_t \Pi_{1}U_{k,1} - \Pi_{1}A(\nabla) U_{k+1,1}  = \Pi_{1}F(U_{a})_{k,1}.
\ee
By the induction assumption \eqref{Un1-form} with $n=k, k+1$ and using \eqref{PiA-Pi12}, we deduce from \eqref{wkb-n-3} that
\be\label{wkb-n-4}
2 i \d_t   \big(g_{k} - \frac{i}{2} \d_{t} g_{k-2}\big) - \Delta  g_{k}  + f(u_{a})_{k,1} = 0 \ \iff \  2 i \d_t   g_{k} +  \d_{tt} g_{k-2} - \Delta  g_{k}  + f(u_{a})_{k,1} = 0,
\ee
which is exactly \eqref{Schro-gn} with $n = k$.

Applying $L_{1}^{-1}$ to \eqref{wkb-n-2} gives
\ba\label{wkb-n-5}
(1-\Pi_{1}) U_{k+2,1}  = -\d_t L_{1}^{-1} U_{k,1}  + L_{1}^{-1} A(\nabla) U_{k+1,1}  + L_{1}^{-1} F(U_{a})_{k,1}.
\ea
By induction assumption \eqref{Un1-form} with $n=k$ and the fact $L_{1}^{-1}\Pi_{1} = \Pi_{1} L_{1}^{-1} = 0$ , we have
\ba\label{wkb-n-6}
 -\d_t L_{1}^{-1} U_{k,1}    =  -\d_t L_{1}^{-1} (1-\Pi_{1})U_{k,1} =    -\d_t L_{1}^{-1} \bp  \nabla g_{k-1} \\  \frac{1}{2} \d_{t}g_{k-2} \\ \frac{i}{2} \d_{t}g_{k-2} \ep  = -   \bp - i \d_{t}\nabla g_{k-1} \\  -\frac{i}{4} \d_{tt}g_{k-2} \\ \frac{1}{4} \d_{tt}g_{k-2} \ep.
 \nn
\ea
By induction assumption \eqref{Un1-form} with $n=k+1$ we have
\ba\label{wkb-n-7}
 L_{1}^{-1} A(\nabla) U_{k+1,1}  =   L_{1}^{-1} A(\nabla) \left[ g_{k+1} e_{+} + \bp \nabla g_{k} \\ \d_{t}g_{k-1} \\ 0 \ep \right]  = \bp \nabla g_{k+1} \\ 0 \\ 0 \ep  + \bp - i \nabla \d_{t} g_{k-1} \\  - \frac{i}{4} \Delta  g_{k}\\ \frac{1}{4} \Delta  g_{k} \ep.
 \nn
 \ea
Direct computation gives
\ba\label{wkb-n-8}
 L_{1}^{-1} F(U_{a})_{k,1}= L_{1}^{-1} \bp 0_d \\ - f(u_{a})_{k,1} \\ 0 \ep  & =  - \bp 0_d \\  -\frac{i}{4} f(u_{a})_{k,1} \\  \frac{1}{4} f(u_{a})_{k,1} \ep.
 \ea
From \eqref{wkb-n-5}--\eqref{wkb-n-8}, together with \eqref{wkb-n-4}, we can deduce
\ba\label{wkb-n-9}
(1-\Pi_{1}) U_{k+2,1}  & = \bp \nabla g_{k+1} \\ 0 \\ 0 \ep  + \bp 0_d \\ -\frac{i}{4}  \big( - \d_{tt}g_{k-2} +  \Delta  g_{k} - f(u_{a})_{k,1}\big)\\ \frac{1}{4}  \big( -\d_{tt}g_{k-2} +  \Delta  g_{k} - f(u_{a})_{k,1}\big) \ep \\
& = \bp \nabla g_{k+1} \\ 0 \\ 0 \ep  + \bp 0_d \\ -\frac{i}{4}  \big( 2 i \d_{t} g_{k}\big )\\ \frac{1}{4}  \big( 2 i \d_{t} g_{k}\big) \ep =  \bp  \nabla g_{k+1}  \\ \frac{1}{2} \d_{t}g_{k} \\\frac{i}{2} \d_{t}g_{k} \ep,
\nn
\ea
which is exactly the first result in \eqref{Un1-form} with $n=k+2$. Moreover,  the second result in \eqref{Un1-form} with $n=k+2$ concerning the form of  $\Pi_{1} U_{k+2,1} $ follows from the structure of $\ker L_{1}$ in \eqref{intr-pola-0}. The proof is thus completed by induction principal.
\end{proof}

\subsubsection{$U_{n,p}$ with $p$ large.}
For $p$ large, we can show that $U_{n,p} = 0$. More precisely, we have
\begin{proposition}\label{prop-Unp-p-large}
For each $n = 0,1,2,\cdots, K_{a}+2$, let
\be\label{p(n)-large}
p(n) := n+1  \ \mbox{for $n$ even}, \quad p(n) := n \  \mbox{for $n$ odd}.\nn
\ee
Then $U_{n,p} = 0$ for $|p|\geq p(n)+1$. Thus the harmonic set of order $n$ is
\be\label{Hn-form}
{\cal{H}}_{n} = \{p\in \Z_{\rm odd}: |p| \leq p(n)\}.
\ee
\end{proposition}

\begin{proof}  We have shown $U_{n,p} = 0$ for $|p|\geq p(n)+1$ with $n=0,1,2,3,4$. Let $k\geq 3.$ By induction we suppose  $U_{n,p} = 0$ for $|p|\geq p(n)+1$ with $n=0,1,2,\cdots, k+1$.  Now we show $U_{k+2,p} = 0$ for $|p|\geq p(k+2)+1$.  The equation $\Phi_{k,p} = 0$ reads
\ba\label{p(n)-large-1}
 L_{p} U_{k+2,p}  = - \d_t U_{k,p} + A(\nabla) U_{k+1,p} + F(U_a)_{k,p}.\nn
\ea
Clearly
$$
 p(k+2)+1\geq p(k+1)+1\geq p(k)+1.
$$
Then for each $|p|\geq p(k+2)+1$, there holds $U_{k,p}  = U_{k+1,p}  = 0$ and therefore
\ba\label{p(n)-large-2}
 L_{p} U_{k+2,p}  =  F(U_a)_{k,p} = ( 0_d^{\rm T}, -f(u_{a})_{k,p}, 0)^{\rm T}.\nn
\ea
Recall
\ba \label{p(n)-large-3}
 f(u_{a})_{k,p} = \sum_{n_{1}+ n_{2} + n_{3} = k} \; \sum_{p_{1}+p_{2} + p_{3} = p} u_{n_{1},p_{1}} u_{n_{3},p_{1}} u_{n_{3},p_{1}}.\nn
\ea

If $k+2$ is odd, so is $k$, and  $p(k+2) = k+2$,  $p(k) = k$. For any nonnegative integers $n_{1}, n_{2}, n_{3}$ such that $n_{1} + n_{2} + n_{3} = k$, at least one of $n_{i}, i=1,2,3$ is odd. Without loss of generality, we may assume $n_{1}$ is odd. Then by induction assumption, in order to make sure $f(u_{a})_{n,p} \neq 0$, necessarily one needs
$$
|p_{1}| \leq n_{1}, \quad |p_{2} | \leq n_{2} +1, \quad |p_{3} |\leq n_{3} +1,
$$
which ensures
$$
|p| =  |p_{1} + p_{2} +p_{3} | \leq |p_{1} | + |p_{2}| + | p_{3} |  \leq  n_{1}+ n_{2} + n_{3} +2 = k +2.
$$
Thus for each $|p|\geq p(k+2) +1 = k+2 +1$, there holds $f(u_{a})_{k,p} = 0$ and thus $U_{k+2,p} = 0.$

The result when $k+2$ is even can be derived similarly.   
\end{proof}

\subsubsection{Complete form of $U_{n}$}  Now we are ready to give the complete form of the WKB solutions. We have shown the specific form of $U_{n, \pm1}$ in Proposition \ref{prop-Un1}.  Now we focus on the form of $U_{n,p}$ with $p\geq 3$. By \eqref{Unp-p-even} and Proposition \ref{prop-Unp-p-large}, it is sufficient to consider $U_{n,p}$ with $p\in \{3,5,\cdots, p(n)\}$. In the following proposition, we give the specific forms of WKB solutions together with the induction relations, both of which are important to derive the specific growth in time estimates of $U_{n,p}$ in Section \ref{sec:WKB-est-defo}.
\begin{proposition}\label{gn=0-n-odd}
For each $n\in \{0,1,2,\cdots, K_{a}\}\cap \Z_{\rm odd}$, we impose zero initial datum $g_{n}(0,\cdot) = 0$.  Then there hold the following  assertions:

(i) For each $n\in \{0,1,2,\cdots, K_{a}\}\cap \Z_{\rm odd}$, $ g_{n} = 0$.

(ii) For each $n\in \{0,1,2,\cdots, K_{a} + 2\}\cap \Z_{\rm even}$ and for each $p\in \{3,5,\cdots, n+1\}$, $U_{n,p}$ has the form
\ba\label{Unp-even-1}
U_{n,p} & = \bp 0_d \\ v_{n,p} \\ u_{n,p} \ep = \bp 0_d \\  \d_{t} u_{n-2,p} + i p u_{n,p} \\ u_{n,p} \ep, \\
 u_{n,p} & = \frac{1}{1-p^{2}} \Big(- f(u_{a})_{n-2,p} - \d^{2}_{t} u_{n-4,p} + \Delta  u_{n-2,p} -  2ip \d_{t} u_{n-2,p} \Big) .
\ea
Then for each $n\in \{0,1,2,\cdots, K_{a} + 2\}\cap \Z_{\rm even}$, $U_{n}$ is of the form
\ba\label{Unp-even}
U_{n}  = e^{i\th} \left[g_{n} e_{+} + \bp 0_d\\ \d_{t} g_{n-2} \\ 0 \ep  \right] +  \sum_{p=3}^{n+1}e^{ip\th} U_{n,p}+ c.c.,\nn
\ea
where $U_{n,p}, p\geq 3$ satisfies \eqref{Unp-even-1}.

(iii) For each $n\in \{1,2,\cdots, K_{a} + 2\}\cap \Z_{\rm odd}$ and for each $p\in \{3,5,\cdots, n\}$, $U_{n,p}$ has the form
\be\label{Unp-odd-1}
U_{n,p} = \bp w_{n,p} \\ 0 \\ 0 \ep = \bp \nabla u_{n-1,p} \\ 0 \\ 0 \ep.
\ee
Then for each $n\in \{1,2,\cdots, K_{a} + 2\}\cap \Z_{\rm odd}$, $U_{n}$ is of the form
\ba\label{Unp-odd}
U_{n}  =   e^{i\th} \left[g_{n} e_{+} + \bp \nabla g_{n-1} \\ 0 \\ 0 \ep  \right]   +  \sum_{p=3}^{n}e^{ip\th} U_{n,p}+ c.c.,\nn
\ea
where $U_{n,p}, p \geq 3$ satisfies \eqref{Unp-odd-1}.

(iv) There holds $u_{n,1} = g_{n}$ for each  $n\in \{0, 1,2,\cdots, K_{a} + 2\}.$
 
\end{proposition}

\begin{proof} Assertion (iv) follows directly from Assertions (i), (ii) and (iii). We have shown in Section \ref{sec-wkb-e-1} that $g_{1}  = 0$ if the initial datum $g_{1}(0,\cdot) = 0$.  As a result, by  \eqref{U1-form}, \eqref{U2-form}, \eqref{U3-form}, \eqref{wkb-23-2} and \eqref{wkb-25}, together with the fact $g_{1} = 0$, we know that: Assertion (i) holds for $n=1$;  Assertion (ii) holds for $n=0,2,4$; and Assertion (iii) holds for $n=1,3$.

Let $k\in \{1,2,\cdots, K_{a}\} \cap \Z_{\rm odd}$. By induction we assume that: Assertion (i) holds for $n= 1, 3, \cdots, k$;  Assertion (ii) holds for $n = 0,2,\cdots, k+3$; and Assertion (iii) holds for $n = 1,3,\cdots, k+2$.  Then when $n=k+2$, the Schr\"odinger equation \eqref{Schro-gn}  becomes
\be\label{Schro-g-k+2}
 2 i \d_t   g_{k+2}  - \Delta  g_{k+2}  + f(u_{a})_{k+2,1} = 0,  \ \ \mbox{with} \  f(u_{a})_{k+2} = \sum_{n_{1}+n_{2}+n_{3} = k+2} u_{n_{1}} u_{n_{2}} u_{n_{3}}.
\ee
Since $k+2$ is odd, to have $n_{1}+n_{2}+n_{3} = k+2$, at least one of $n_{1}, n_{2}, n_{3}$ is odd.  If $n_{i} $ is odd and $n_{i} \leq k$, by induction assumptions for Assertions (i) and (iii), we know that $g_{n_{i}} = 0$ and $u_{n_{i}} = 0$. Thus, together with induction assumption for Assertion (iii) with $n=k+2$,
\be\label{f(ua)-k+2}
f(u_{a})_{k+2} = 3 u_{0}^{2} u_{k+2} = 3\l \left[\big(e^{i\th} g_{0} +e^{-i\th} \bar g_{0} \big)^{2} \big(  e^{i\th} g_{k+2}  + e^{-i\th} \bar g_{k+2}   \big)\right].
\ee
This implies
\be\label{f-k+2-1}
f(u_{a})_{k+2,1} = 3\l\left(2|g_{0}|^{2} g_{k+2} + g_{0}^{2} \bar g_{k+2} \right).\nn
 \ee
Then, with initial datum $g_{k+2}(0, \cdot) = 0$, the equation \eqref{Schro-g-k+2} admits a unique trivial solution $g_{k+2} = 0$. This corresponds to Assertion (i) with $n=k+2$.

\medskip

We next prove Assertion (iii) with $n=k+4$.  By \eqref{Un1-form} in Proposition \ref{prop-Un1} and $g_{k+2} = 0$,  we know that
\be\label{U-k+4-1}
U_{k+4,1} = g_{k+4} e_{+} + \bp \nabla g_{k+3} \\ \d_{t} g_{k+2} \\ 0 \ep =  g_{k+4} e_{+} + \bp \nabla g_{k+3} \\ 0 \\ 0 \ep.\nn
\ee
For $p\geq 3$, the equation $\Phi_{k+2,p} = 0$ reads
\ba\label{Lp-k+4-p}
 L_{p} U_{k+4,p}  = - \d_t U_{k+2,p} + A(\nabla) U_{k+3,p} + F(U_a)_{k+2,p}.
\ea
By induction assumptions, we know that
$$
\d_{t}U_{k+2,p} = \bp \d_{t} w_{k+2,p} \\ 0 \\ 0 \ep, \quad A(\nabla) U_{k+3,p}  = A(\nabla)  \bp 0_d \\ v_{k+3,p} \\ u_{k+3,p} \ep = \bp  \nabla v_{k+3,p} \\ 0 \\ 0  \ep.
$$
Since $k+2$ is odd, we know $f(u_{a})_{k+2}$ takes the form \eqref{f(ua)-k+2}. Together with the result $g_{k+2} = 0$, one has
\ba\label{FUa-k+2}
F(U_{a})_{k+2} =( 0_d^{\rm T}, - f(u_{a})_{k+2}, 0 )^{\rm T} = 0.
\ea
Thus, by the form of $L_{p}^{-1}$ in  \eqref{Lp-p>2} and by \eqref{Lp-k+4-p}--\eqref{FUa-k+2}, we have
\be\label{U-k+4-3}
U_{k+4,p} = \bp w_{k+4,p} \\ 0 \\ 0 \ep  \ \mbox{with} \ w_{k+4,p} = - \frac{i}{p} \big(  \d_{t} w_{k+2,p} + \nabla v_{k+3,p} \big).
\ee
By induction assumptions, we know
$$
w_{k+2,p} = \nabla u_{k+1,p}, \quad v_{k+3,p} = \d_{t} u_{k+1,p} + ip u_{k+3,p}.
$$
Together with \eqref{U-k+4-3}, we obtain $w_{k+4,p} = \nabla u_{k+3,p}$ and thus prove Assertion (iii) with $n=k+4$.

\medskip

Next, by Assertion (iii) with $n=k+4$ we have shown in the previous step, similar as the argument of showing $g_{k+2} = 0$, we can prove $g_{k+4} = 0$.

\medskip

We finally prove Assertion (ii) with $n=k+5$.  By \eqref{Un1-form} in Proposition \ref{prop-Un1} and $g_{k+4} = 0$,  we can show that
\be\label{U-k+3-1}
U_{k+5,1} = \bp \nabla g_{k+4} \\  \frac{1}{2} \d_{t}g_{k+3} \\ \frac{i}{2} \d_{t}g_{k+4} \ep + \big(g_{k+5} - \frac{i}{2} \d_{t} g_{k+3}\big) e_{+} = g_{k+5} e_{+} + \bp 0_d \\ \d_{t} g_{k+3} \\ 0 \ep.\nn
\ee
For $p\geq 3$, the equation $\Phi_{k+3,p} = 0$ reads
\ba\label{Lp-k+3-p}
 L_{p} U_{k+5,p}  = - \d_t U_{k+3,p} + A(\nabla) U_{k+4,p} + F(U_a)_{k+3,p}.\nn
\ea
By induction assumptions and Assertion (iii) with $n=k+4$ we have shown in the previous step, we know that
\ba\label{Lp-k+3-p-2}
&\d_{t}U_{k+3,p} = \bp 0_d \\ \d_{t} v_{k+3,p} \\ \d_{t} u_{k+3,p} \ep
=  \bp 0_d \\ \d_{t}^2 u_{k+1,p} + ip \d_{t}u_{k+3,p} \\ \d_{t} u_{k+3,p} \ep
= \d_{t}^{2} u_{k+1,p}  \bp 0_d \\1  \\ 0  \ep + \d_{t} u_{k+3,p} \bp 0_d \\ ip  \\ 1  \ep,  \\
&A(\nabla) U_{k+4,p}  = A(\nabla) \bp w_{k+4,p} \\ 0 \\ 0 \ep  = \Delta  u_{k+3,p} \bp 0_d \\1  \\ 0  \ep , \ \ F(U_a)_{k+3,p} = -f(u_{a})_{k+3,p}\bp 0_d \\ 1 \\0 \ep.
\ea
By the form of $L_{p}^{-1}$ in  \eqref{Lp-p>2}, we deduce from \eqref{Lp-k+3-p-2} that Assertion (ii) holds with $n=k+5$. The proof is completed.
\end{proof}

We introduce the notation
\ba\label{def-dn}
\d_{x}^{n} : = (\d_{x}^{\a})_{\a\in \N^{d}, |\a| \leq n} = (\d_{x_{1}}^{\a_{1}} \d_{x_{2}}^{\a_{2}}  \cdots \d_{x_{d}}^{\a_{d}} )_{\a = (\a_{1}, \a_{2}, \cdots, \a_{d})\in \N^{d}, |\a| \leq n} , \ \mbox{$n \in \Z_{+}$}. \nn
\ea
We use $\mathcal{P}_{n}$ to denote a polynomial of  multiple variables of some degree depending on $n$, whose form may differ from line to line.

A direct corollary of Propositions \ref{prop-Un1} and \ref{gn=0-n-odd} is the following
\begin{corollary}\label{prop-Unp-gn}
For each  $n\in \{0,1,2,\cdots, K_{a}\}\cap \Z_{\rm odd}$, we impose zero initial datum $g_{n}(0,\cdot) = 0$.  Then for each $n\in \{0,1,2,\cdots, K_{a}+2\}\cap \Z_{\rm even}$ and each $p\geq 3$,
 \be\label{Unp-unp-even}
u_{n,p} = \calP_{n}(\d^{n-2}_{x} (g_{0}, \bar g_{0}), \d^{n-4}_{x} (g_{2}, \bar g_{2}) \cdots, g_{n-2}, \bar g_{n-2}).\nn
\ee
Consequently, we have:

(i) For each $n\in \{2,\cdots, K_{a}+2\}\cap \Z_{\rm even}$ and each $p\geq 3$,
\be\label{Unp-gn-even}
U_{n,p} = \calP_{n}\left(\d^{n-2}_{x} (g_{0}, \bar g_{0}), \d^{n-4}_{x} (g_{2}, \bar g_{2}) \cdots, g_{n-2}, \bar g_{n-2}\right).\nn
\ee

(ii) For each $n\in \{2,\cdots, K_{a}+2\}\cap \Z_{\rm odd}$ and each $p\geq 3$,
\be\label{Unp-gn-odd}
U_{n,p} =\calP_{n}\left(\d^{n-2}_{x} (g_{0}, \bar g_{0}), \d^{n-4}_{x} ( g_{2}, \bar g_{2}), \cdots, \d_{x} (g_{n-3}, \bar g_{n-3})\right).\nn
\ee
 
\end{corollary}
 
The proof of  Corollary \ref{prop-Unp-gn} is rather straightforward by induction argument and Propositions \ref{prop-Un1} and \ref{gn=0-n-odd}, which is omitted here for brevity.

\subsection{Initial data of $U_{n,p}$}

We have imposed initial datam $g_{0}(0,\cdot) =\frac{\phi - i \psi}{2}$ in \eqref{wkb-ini-g0} and $g_{1}(0,\cdot) = 0$ such that
$$
 \big(U_{0} + \e U_{1}\big)(0,\cdot) = (\e \nabla^{\rm T} \phi, \psi, \phi) = U(0,\cdot),
$$
where $U(0,\cdot)$ is the initial datum of the exact solution $U$, see \eqref{ini-data00}.

To ensure the initial difference $U(0,\cdot) - U_{a}(0,\cdot)$ to be small, we have chosen $g_{2}(0,\cdot)$ as in \eqref{g20-2} such that $U_{2}(0,\cdot) = 0$. Moreover, with the choice of initial datum $g_{2}(0,\cdot)$ in \eqref{g20-2} and the choice  $g_{3}(0,\cdot) = 0$, we can deduce from \eqref{U3-form} that
$
U_{3}(0,\cdot) = 0.
$
By induction, we can show that with proper choice of initial data for $g_{n}$, we can ensure $U_{n}(0,\cdot) = 0$ for all $n=0,1, \cdots, K_{a}, K_{a}+1$:
\begin{proposition}\label{ini-Un=0} Let $g_{n}(0,\cdot) = 0$ for all $n\in \{1,2,\cdots, K_{a}\}\cap \Z_{\rm odd}$. For each $n=\{2,\cdots, K_{a}\}\cap \Z_{\rm even}$, there exists a polynomial $\calP_{n}(\d_x^{n} \phi, \d_x^{n} \psi)$ such that by imposing initial data
\ba\label{ini-gn}
 g_{n}(0,\cdot) = \calP_{n}(\d_x^{n} \phi, \d_x^{n} \psi),\nn
\ea
there holds
\ba\label{Un=0}
U_{n}(0,\cdot) = 0 \quad \mbox{for all $n\in \{2,3,\cdots, K_{a}, K_{a}+1\}$}.
\ea

\end{proposition}

\begin{proof} We know from Proposition \ref{gn=0-n-odd}  that $g_{n} = 0$ for all $n$ odd.   We have shown the results \eqref{Un=0} for $n=1,2,3$.  In particular $\calP_{1} = g_{0}(0,\cdot)$ and $\calP_{2} = g_{2}(0,\cdot)$ are given in \eqref{wkb-ini-g0} and \eqref{g20-2}, respectively. Let integer $2\leq k \leq K_{a}$ be odd. By induction we  assume that the results \eqref{Un=0}   hold for all $n\leq k$. We shall show that there exists some polynomial $\calP(\d_x^{k+1} \phi, \d_x^{k+1}\psi)$  such that by imposing initial data
$
g_{k+1}(0,\cdot) =\calP (\d_x^{k+1} \phi, \d_x^{k+1} \psi),
$
there holds
$$
U_{k+1}(0,\cdot) = U_{k+2}(0,\cdot) = 0.
$$

By Assertion (ii) in Proposition \ref{gn=0-n-odd},  we know that $U_{k+1}$ is of the form
\ba\label{Un=0-1}
U_{k+1}  = e^{i\th} \left[g_{k+1} e_{+} + \bp 0_d \\ \d_{t} g_{k-1} \\ 0 \ep  \right] +  \sum_{p=3}^{k+2}e^{ip\th}  \bp 0_d \\  \d_{t} u_{k-1,p} + i p u_{k+1,p} \\ u_{k+1,p} \ep + c.c.\,.\nn
\ea
To ensure $U_{k+1}(0,\cdot) = 0$, it suffices to impose
\be\label{gk+1-0-1}
\left\{
\begin{aligned} & i g_{k+1}(0,\cdot) + \d_{t} g_{k-1}(0,\cdot)  +  \sum_{p=3}^{k+2} (\d_{t} u_{k-1,p} + i p u_{k+1,p}) (0,\cdot)+ c.c. = 0, \\
& g_{k+1}(0,\cdot)  +  \sum_{p=3}^{k+2} u_{k+1,p} (0,\cdot)+ c.c. = 0,
\end{aligned}
\right.
\nn
\ee
which is equivalent to
\ba\label{gk+1-0-2}
{Im}(g_{k+1}(0,\cdot)) & = Re\big[ \d_{t} g_{k-1}(0,\cdot)  +   \sum_{p=3}^{k+2} (\d_{t} u_{k-1,p} + i p u_{k+1,p}) (0,\cdot) \big] ,\\
{Re}(g_{k+1}(0,\cdot)) & = - Re \sum_{p=3}^{k+2} u_{k+1,p} (0,\cdot) .
\ea
By Proposition \ref{gn=0-n-odd}, Corollary \ref{prop-Unp-gn}, and the induction assumption, together with the  Schr\"odinger equations for $g_{n}, n\leq k-1$,  we can deduce from \eqref{gk+1-0-2} that the initial data $g_{k+1}(0,\cdot)$ can be written as a polynomial in  $(\d_x^{k+1} \phi, \d_x^{k+1} \psi)$.  

Next, we show that with such a choice of initial datum for $g_{k+1}(0,\cdot) $ in \eqref{gk+1-0-2}, there holds automatically $U_{k+2}(0,\cdot) = 0$. Indeed, by Proposition \ref{gn=0-n-odd}, we have
\ba\label{Lp-k+2-p-3}
U_{k+2}  =   e^{i\th}  \bp \nabla g_{k+1} \\ 0 \\ 0 \ep    +  \sum_{p=3}^{n}e^{ip\th} \bp \nabla u_{k+1,p} \\ 0 \\ 0 \ep +  c.c.,\nn
\ea
which, together with the second equation in $\eqref{gk+1-0-2}$,  ensures  $U_{k+2}(0,\cdot) =0$.

\end{proof}

\section{Regularity of WKB solutions} \label{sec:reg-wkb}
In this section, we shall show the regularity of the component $U_{n,p}$ of the WKB solutions and further prove  Theorems \ref{thm:app-defo} and \ref{thm:app-defo}.

\subsection{Regularity of WKB solutions: focusing case}

Now we can summarize the results obtained in Section \ref{sec-wkb-e-n} to give the regularity and precision of the WKB solutions.  As shown in Corollary \ref{prop-Unp-gn}, each component $U_{n,p}$ can be written as a polynomial of  $g_{0}, g_{2}, \cdots, g_{n}$ and their derivatives up to order $n-2$. Hence the regularities of $U_{n,p}$ are determined by the regularities of $g_{0}, g_{2}, \cdots, g_{n}$.  We first consider the focusing case with $\l<0$.

\begin{proposition}\label{prop-reg-gn}
Let $\l<0$. Let  $K_{a} \in \Z_{\geq 0}\cap \Z_{\rm even}$. Suppose $(\phi, \psi) \in H^{s}$ with $s>s_{a}:= 2K_{a} + 4 + \frac{3}{2}$. Let $T^{*}$ be the existence time to the Cauchy problem of the cubic Schr\"odinger equation \eqref{eq-g-cub} given in Proposition \ref{prop-Sch-local}.  Let $g_{n} \equiv 0$ for each $n\in \{0,1,2,\cdots, K_{a}\}\cap \Z_{\rm odd}$. For each $n\in \{2,\cdots, K_{a}\}\cap \Z_{\rm even}$, let $g_{n}$ be the solution to the linear Schr\"odinger equation \eqref{Schro-gn} with initial datum $g_{n}(0,\cdot)$ chosen as in Proposition \ref{ini-Un=0}. Then
\ba\label{reg-gn-0}
g_{n}\in C([0,T^{*}); H^{s-2n}(\R^{3})), \quad \mbox{for each $n\in \{2,\cdots, K_{a}\}\cap \Z_{\rm even}$}.\nn
\ea
Furthermore,  by choosing $g_{K_{a}+1} = g_{K_{a}+2} = 0$, there exists a WKB solution $U_{a}$ of the form \eqref{def-wkb} with $U_{n,p}$ satisfying the properties stated in Propositions \ref{prop-Un1}, \ref{prop-Unp-p-large}, \ref{gn=0-n-odd}, \ref{prop-Unp-gn}, and \ref{ini-Un=0}. In particular, the amplitudes $U_{n,p}$ satisfy the estimates:
\ba\label{reg-un}
& U_{n,1} \in C([0,T^{*}); H^{s-2n}(\R^{3})), \quad \mbox{for each $n=0,1,2, \cdots, K_{a}$,}\\
& U_{K_{a}+1,1} \in C([0,T^{*}); H^{s-2K_{a} -1}(\R^{3})), \quad  U_{K_{a}+2,1} \in C([0,T^{*}); H^{s-2K_{a} -2}(\R^{3})),\\
& U_{n,p} \in C([0,T^{*}); H^{s-2n+2}(\R^{3})), \quad \mbox{for each $n=0,1,2, \cdots, K_{a}+2$ and each $|p|\geq 3$,}
\ea
and the initial difference satisfies
\ba\label{ini-ua-u}
\|U(0,\cdot) - U_{a}(0, \cdot)\|_{H^{s-2K_{a}-2}} \leq C_{0} \e^{K_{a}+2}.
\ea

\end{proposition}

\begin{proof}[Proof of Proposition \ref{prop-reg-gn}]

First of all, as shown in Proposition \ref{prop-Sch-local},  there admits a unique solution $g_{0} \in C([0,T^{*}), H^{s}(\R^{3}))$ such that
\ba\label{g0-bdd-1}
\|g_{0}\|_{L^{\infty}(0,t; H^{s})} \leq D(t), \quad \forall \, t\in [0,T^{*}).\nn
\ea

\medskip

By Proposition \ref{ini-Un=0} and the fact that $H^{s}(\R^{3})$ is a Banach algebra when $s>d/2$,  one has
\ba\label{ini-gn-reg}
\|g_{n}(0,\cdot) \|_{H^{s-n}}\leq C_0, \ \mbox{for each even $n\in \{2,\cdots, K_{a}\}\cap \Z_{\rm even}$}. \nn
\ea
In particular, $g_{2}$ satisfies the linear Schr\"odinger equation \eqref{wkb-21-2} with initial datum \eqref{g20-2}. Using Duhamel formula gives for each $t\in (0,T^{*})$ that
\ba\label{g2-duha-1}
g_{2}(t,\cdot)  =  e^{-\frac{it\Delta }{2}} g_{2}(0,\cdot) + \int_{0}^{t}   e^{-\frac{i(t-s)\Delta }{2} } \frac{i}{2} \big( \d_{tt} g_{0} + f(u_a)_{2,1}\big)(t', \cdot)  \dd t' .
\ea
Using the equation \eqref{wkb-U01-4} in $g_{0}$ implies that $\d_{tt} g_{0} \in C([0,T^{*}), H^{s-4}(\R^{3}))$. Then by the form of $f(u_a)_{2,1}$ in \eqref{wkb-21-0} we deduce that
\ba\label{g2-duha-2}
\| g_{2}(t,\cdot)\|_{H^{s-4}}  & \leq   \| g_{2}(0,\cdot)\|_{H^{s-4}}  + \int_{0}^{t} \frac{1}{2}  \big(\|\d_{tt} g_{0}(t', \cdot)\|_{H^{s-4}}  + \|f(u_a)_{2,1}(t', \cdot)\|   \big) \dd t' \\
& \leq C_{0}+ \int_{0}^{t} D(t') \big(1  + \| g_{2}(t', \cdot)\|_{H^{s-4}}\big)\dd t' .
\ea
Applying Gronwall's inequality to \eqref{g2-duha-2} implies that
\ba\label{g2-bdd-1}
\|g_{2}\|_{L^{\infty}(0,t; H^{s-4})} \leq D(t), \quad \forall \, t\in [0,T^{*}).
\ea
Then $g_{2}\in C([0,T^{*}); H^{s-4}(\R^{3}))$ where the continuity in time variable follows from the Duhamel formula \eqref{g2-duha-1}. Note that the generic continuous increasing function $D(t)$ in \eqref{g2-bdd-1} may differ from the one in \eqref{g2-duha-2}.

\medskip

Recall the Schr\"odinger equation \eqref{Schro-gn} for $g_{n}$ with $n$ even:
 \be\label{Schro-gn-recall}
 2 i \d_t   g_{n} +  \d_{tt} g_{n-2} - \Delta  g_{n}  + f(u_{a})_{n,1} = 0.\nn
\ee
By \eqref{Fnp} concerning the form of $f(u_{a})_{n,p}$, one has
\ba \label{f-ua-n1}
 f(u_{a})_{n,1} = \l (3 u_{0,1}^{2} u_{n,-1} + 6u_{0,1}u_{0,-1}  u_{n,1}) + f(u_{a})_{n,1}^{(r)}\nn
 \ea
with
 \ba \label{f-ua-n1-2}
 f(u_{a})_{n,1}^{(r)} :=  \l \sum_{n_{1}+ n_{2} + n_{3} = n; n_{1}, n_{2}, n_{3} \leq n-1}\; \sum_{p_{1}+p_{2} + p_{3} = 1} u_{n_{1},p_{1}} u_{n_{2},p_{2}} u_{n_{3},p_{3}}.\nn
  \ea
By Proposition \ref{gn=0-n-odd} we know that $u_{n,1} = g_{n}$ for each $n$. Thus,
\ba \label{f-ua-n1-3}
 f(u_{a})_{n,1} =  \l (3 g_{0}^{2} \bar g_{n} + 6 |g_{0}|^{2} g_{n}) +\l \sum_{n_{1}+ n_{2} + n_{3} = n; n_{1}, n_{2}, n_{3} \leq n-1}\; \sum_{p_{1}+p_{2} + p_{3} = 1} u_{n_{1},p_{1}} u_{n_{2},p_{2}} u_{n_{3},p_{3}}.
 \ea
  By Proposition  \ref{prop-Unp-gn}, we know for each $k\leq n-1$ and each $p\in \calH_{k},  \ p\geq 3$, $u_{k,p}$ is a polynomial in $g_{0}, g_{2}, \cdots, g_{k-2}$.  Thus, similarly as \eqref{g2-duha-1}--\eqref{g2-bdd-1}, by induction argument we have $g_{n} \in C([0,T^{*}); H^{s-2n}(\R^{3}))$,  for each $n\in \{2, \cdots, K_{a}\}\cap \Z_{\rm even}. $

 \medskip

 By choosing $g_{_{K_{a}+1}} = g_{_{K_{a}+2}} = 0$ and applying Propositions \ref{prop-Un1}, \ref{prop-Unp-p-large}, \ref{gn=0-n-odd}, \ref{prop-Unp-gn},  \ref{ini-Un=0} and \ref{prop-reg-gn},  it is rather straightforward to construct a WKB solution $U_{a}$ of the form \eqref{def-wkb} satisfying \eqref{reg-un}. By Proposition \ref{ini-Un=0}, the initial perturbation is
 \ba\label{ini-pert-defo}
U_{a}(0,\cdot) - U(0,\cdot) =  \e^{K_{a}+2} U_{K_{a}+2}(0,\cdot)
 \ea
which satisfies \eqref{ini-ua-u}.
\end{proof}

\subsection{Regularity of WKB solutions: defocusing case}\label{sec:WKB-est-defo}

For the defocusing case with $\l>0$, a global-in-time WKB solution can be constructed. Moreover, we can derive specific growth rates in time for all $U_{n,p}$ which are important to obtain the linear-in-time growth of error estimates of the form $(1+t) \e^{2}$ in Section \ref{sec:sta-WKB}.

\begin{proposition}\label{prop-reg-gn-defo}
Let $\l>0$. Let $K_{a} \in \Z_{\geq 0}\cap \Z_{\rm even}$. Suppose $(\phi, \psi)$ satisfy \eqref{ini-ass1}--\eqref{ini-ass2} with $s> s_{a}:= 2  K_{a} +4 + \frac{3}{2}$.  Let $g_{n} \equiv 0$ for each $n\in \{0,1,2,\cdots, K_{a}\}\cap \Z_{\rm odd}$ and $g_{_{K_{a}+1}}= g_{_{K_{a}+2}} =0$. For each $n\in \{2,3,\cdots, K_{a}\}\cap \Z_{\rm even}$, let $g_{n}$ be the solution to the linear Schr\"odinger equation \eqref{Schro-gn} with initial datum $g_{n}(0,\cdot)$ chosen as in Proposition \ref{ini-Un=0}. Then
\ba\label{reg-gn-0-defo}
\d_{t}^{j} g_{n}\in C([0,\infty); H^{s-2n-2j}(\R^{3})), \quad \mbox{for each $n\in \{2,3,\cdots, K_{a}\}\cap \Z_{\rm even}, \ j\leq \frac{K_{a}}{2} - n$}.
\ea
Moreover, $g_{0}$ satisfies the properties given in Proposition \ref{prop-Sch-global}, and $g_{n},  \  n\in \{0, 1,\cdots,K_{a}\}\cap \Z_{\rm even}$ satisfy the following estimates for all $t>0$:
\ba\label{est-g2-1}
 \| \d_{t}^{j}g_{n}(t, \cdot)\|_{H^{s-2n - 2j }(\R^{3})} \leq C_{0} + C_{0}(1 + t)^{n-1}  , \  \mbox{for all $j\leq \frac{K_{a}}{2} - n.$}
 \ea
In addition, for each $ n\in \{ 2,3, \cdots, K_{a}+2\}\cap \Z_{\rm even},  \ p\geq 3$, there holds
\ba\label{reg-unp-defo}
\| \d_{t}^{j} u_{n,p}(t,\cdot)\|_{H^{s-2n+4 -2j }}  \leq   C_{0} + C_{0}(1 + t)^{n-3},  \ \mbox{for all $j\leq \frac{K_{a}}{2} - n + 2.$}
 \ea

\end{proposition}

\begin{proof}[Proof of Proposition \ref{prop-reg-gn-defo}]
 It has been shown in Proposition \ref{prop-Sch-global} that $g_{0} \in L^{\infty}(0,\infty; H^{s}(\R^{3}))$ and $\|g_{0}(t, \cdot)\|_{L^{\infty}(\R^{3})}$ decays as $t^{-\frac{3}{2}}$ as $t\to \infty.$ For each $n\geq 2,$ the equations  \eqref{Schro-gn}  for $g_{n}$ are linear Schr\"odinger equations. By induction, it is rather straightforward to deduce that  $g_{n}\in C([0,\infty); H^{s-2n}(\R^{3}))$, for each $n\in \{2,3,\cdots, K_{a}\}\cap \Z_{\rm even}$, which is exactly \eqref{reg-gn-0-defo}.

 \medskip

 It is left to show the estimates in \eqref{est-g2-1}--\eqref{reg-unp-defo}. This is done by induction.  Firstly, using the fact $g_{0} \in L^{\infty}(0,\infty; H^{s}(\R^{3}))$ and the cubic Schr\"odinger equation \eqref{eq-g-cub} for $g_{0}$, we have
 \ba\label{est-dt-g0-1}
 \|\d_{t}g_{0} (t, \cdot)\|_{H^{s-2}} \leq C \|\Delta  g_{0}(t,\cdot)\|_{H^{s-2}}  + C \|\, |g_{0}(t,\cdot)|^{2} g_{0}(t,\cdot)\|_{H^{s-2}}  \leq C_{0},\nn
 \ea
where we  used the classical estimates in Sobolev spaces:
\ba\label{Sobolev-mul}
&\|uv\|_{H^{s}} \leq C \big(\|u\|_{H^{s}} \|v\|_{L^{\infty}} + \|u\|_{L^{\infty}} \|v\|_{H^{s}} \big), \quad &&\mbox{for all $s>0$},\\
&\|u\|_{L^{\infty}}\leq C \|u\|_{H^{s}}, \quad &&\mbox{for all $s>\frac{3}{2}$}.
\ea
Applying $\d_{t}^{j}$ to the cubic Schr\"odinger equation \eqref{eq-g-cub} for $g_{0}$ and by induction argument, we can obtain
 \ba\label{est-dt-g0-2}
 \|\d_{t}^{j}g_{0} (t, \cdot)\|_{H^{s-2j}}   \leq C_{0}, \quad \mbox{for all  $\ j\leq \frac{K_{a}}{2}.$}
 \ea
This is \eqref{est-g2-1} with $n=0$.

Next, we show the estimate of $g_{2}$. By the equations \eqref{wkb-21-2}--\eqref{g20-2} for $g_{2}$ and the Duhamel formula we have for all $t\in (0,\infty)$,
 \ba\label{g2-duha-de-1}
\| g_{2}(t,\cdot)\|_{H^{s-4}}  & \leq   \| g_{2}(0,\cdot)\|_{H^{s-4}}  + \int_{0}^{t} \frac{1}{2}  \big(\|\d_{tt} g_{0}(t', \cdot)\|_{H^{s-4}}  + \| f(u_a)_{2,1}(t', \cdot)\|   \big) \dd t' \\
& \leq C_{0} + \int_{0}^{t}  \frac{1}{2}\big(C_{0}  + \| f(u_a)_{2,1}(t', \cdot)\|_{H^{s-4}}\big)\dd t' .
\ea
By the form of $f(u_a)_{2,1}$ in \eqref{wkb-21-0} and using Proposition \ref{prop-Sch-global} and \eqref{Sobolev-mul}, we have
\ba\label{est-f21-defo}
\| f(u_a)_{2,1}(t,\cdot)\|_{H^{s - 4}} &\leq  3\l \big(\|g_{0}^{2} \bar g_{2} \|_{H^{s-4}} + 2\| |g_{0}|^{2} g_{2} \|_{H^{s-4}} +  \frac{\l}{8} \| |g_{0}|^{4} g_{0} \|_{H^{s-4}}  \big)\\
&\leq  C_{0} \|g_{0}\|_{H^{s-4}} \|g_{0}\|_{L^{\infty}} \| g_{2} \|_{H^{s-4}} + C_{0} \|g_{0}\|_{H^{s}} \|g_{0}\|_{L^{\infty}}^{4}\\
&\leq   C_0(1+t)^{-\frac{3}{2}} \| g_{2} (t,\cdot)\|_{H^{s-4}} +  C_0(1+t)^{-\frac{3}{2}} .
\ea

\medskip

Then,  from \eqref{g2-duha-de-1} and \eqref{est-f21-defo},  we deduce for all $t\in (0,\infty)$ that
 \ba\label{g2-duha-de-2}
\| g_{2}(t,\cdot)\|_{H^{s-4}}  & \leq   \| g_{2}(0,\cdot)\|_{H^{s-4}}  + \int_{0}^{t} \frac{1}{2}  \big(\|\d_{tt} g_{0}(t',\cdot )\|_{H^{s-4}}  + \|f(u_a)_{2,1}(t',\cdot)\|_{H^{s-4}}   \big) \dd t'\\
& \leq  C_0(1+t) +  \int_{0}^{t} \left[  C_0(1+t)^{-\frac{3}{2}} \| g_{2} (t',\cdot)\|_{H^{s}} +  C_0(1+t)^{-\frac{3}{2}} \right] \dd t'\\
&\leq C_0(1+t)  +      \int_{0}^{t}   C_0(1+t)^{-\frac{3}{2}} \| g_{2} (t',\cdot)\|_{H^{s-4}} \dd t'.
\ea
Applying Gronwall's inequality to \eqref{g2-duha-de-2} gives
 \ba\label{g2-duha-de-3}
\| g_{2}(t,\cdot)\|_{H^{s-4}}  \leq C_0(1+t)+ \int_{0}^{t} C_0(1+t)  (1 + t')^{-\frac{3}{2}} e^{\int_{t'}^{t}  C_{0}(1 + r)^{-\frac{3}{2}}  \dd r}\dd t'  \leq  C_0(1+t) .\nn
 \ea
 Then using the Schr\"odinger equation \eqref{wkb-21-2} for $g_{2}$, we have
 \ba\label{g2-duha-de-8}
&\| \d_{t} g_{2}(t,\cdot)\|_{H^{s-6}}    \leq  C \big(\| \d_{t}^{2} g_{0}(t,\cdot)\|_{H^{s-6}} +  \| \Delta  g_{2}(t,\cdot)\|_{H^{s-6}}  +   \| f(u_{a})_{2,1} (t,\cdot)\|_{H^{s-6}}  \big) \\
&\quad  \leq C \big(\| \d_{t}^{2} g_{0}(t,\cdot)\|_{H^{s-6}} +  \| g_{2}(t,\cdot)\|_{H^{s-4}}  +   \|g_{0} (t,\cdot) \|_{H^{s-6}}^{2}  \| g_{2} \|_{H^{s-6}} + \|g_{0}(t,\cdot) \|_{H^{s-6}}^{5} \big)\\
 &\quad \leq  C_0(1+t) .\nn
\ea
Applying $\d_{t}^{j}$ to the Schr\"odinger equation \eqref{wkb-21-2} for $g_{2}$ and  using induction argument, we can obtain
 \ba\label{g2-duha-de-9}
 \|\d_{t}^{j}g_{2} (t, \cdot)\|_{H^{s-4-2j}}   \leq C_0(1+t) , \quad \mbox{for all $ j \leq \frac{K_{a}}{2} - 2.$}\nn
 \ea
This is \eqref{est-g2-1} with $n=2$.

\medskip

By \eqref{U2-form} and Proposition \ref{prop-Sch-global}, we have
 \ba\label{est-u23-defo}
\| u_{2,3}(t,\cdot)\|_{H^{s}}   \leq \frac{\l}{8}\| g_{0}^{3}(t,\cdot) \|_{H^{s}} \leq C \|g_{0}(t,\cdot)\|_{L^{\infty}}^{2}  \|g_{0}(t,\cdot)\|_{H^{s}}\leq  C_0(1+t) ^{-3} \leq C_{0}.\nn
\ea
Using the estimates \eqref{est-dt-g0-2} for $g_{0}$ gives
 \ba\label{est-u23-defo-2}
\| \d_{t}^{j} u_{2,3}(t,\cdot)\|_{H^{s-2j}}  \leq C_{0},  \quad \mbox{for all  $ j\leq \frac{K_{a}}{2}.$}\nn
\ea
Similarly, by \eqref{wkb-23-2} and \eqref{wkb-25}, we have
 \ba\label{est-u43-defo}
\| \d_{t}^{j} u_{4,3}(t,\cdot)\|_{H^{s-2-2j}}   \leq   C_0(1+t) ,  \quad \| \d_{t}^{j}u_{4,5}(t,\cdot)\|_{H^{s-2j}}  \leq  C_{0}, \ \mbox{for all $ j \leq \frac{K_{a}}{2} - 2$}.\nn
\ea

Up to now, we have shown  \eqref{est-g2-1} for $n=0,2$ and \eqref{reg-unp-defo} for $n=0,2,4$. Let $k\geq 2$ be even. By induction we assume \eqref{est-g2-1} hold all $n\leq k$ and \eqref{reg-unp-defo} hold for all $n \leq k+2$. We shall show \eqref{est-g2-1} holds for $n = k+2$ and \eqref{reg-unp-defo} holds for  $n  =  k+4$.

\medskip

We first show \eqref{est-g2-1} with $n=k+2$ concerning the estimate of $g_{k+2}$. This estimate is the crucial one.  Recall the equation for $g_{k+2}$:
\be\label{Schro-g-k+2-even}
 2 i \d_t   g_{k+2}  - \Delta  g_{k+2}  + \d_{tt} g_{k} +  f(u_{a})_{k+2,1} = 0.
\ee
As in \eqref{f-ua-n1-3}  we can write
\ba \label{f-ua-k+2-1}
 f(u_{a})_{k+2,1} & =   f(u_{a})_{k+2,1}^{(0)}  + f(u_{a})_{k+2,1}^{(r)},  \\
  f(u_{a})_{k+2,1}^{(0)} &: =   3 \l ( g_{0}^{2} \bar g_{k+2} + 2 |g_{0}|^{2} g_{k+2}), \\
  f(u_{a})_{k+2,1}^{(r)} & : = \l \sum_{n_{1}+ n_{2} + n_{3} = k+2; n_{1}, n_{2}, n_{3} \leq k}\; \sum_{p_{1}+p_{2} + p_{3} = 1} u_{n_{1},p_{1}} u_{n_{2},p_{2}} u_{n_{3},p_{3}},
 \ea
where we used the fact that $u_{n} = 0$ for all odd $n$ which follows from Proposition \ref{gn=0-n-odd}.

Similar as the estimate \eqref{est-f21-defo}, we have
\ba\label{f-n1-0-2-k+2}
\| f(u_{a})_{k+2,1}^{(0)}(t, \cdot)\|_{H^{s - 2(k+2)}}  & \leq 3\l \big(\|g_{0}^{2} \bar g_{k+2} \|_{H^{s-2(k+2)}} + 2\| |g_{0}|^{2} g_{k+2} \|_{H^{s-2(k+2)}} \big)\\
& \leq C \|g_{0}\|_{H^{s-2(k+2)}} \|g_{0}\|_{L^{\infty}} \| g_{k+2} \|_{H^{s-2(k+2)}} \\
&\leq  C_0(1+t) ^{-\frac{3}{2}} \| g_{k+2} (t,\cdot)\|_{H^{s-2(k+2)}}.
\ea

Observe that in the sum of $f(u_{a})_{k+2,1}^{(r)}$, at most one of $n_{i}$ is $0$. Then we can write
\ba \label{f-ua-k+2-2}
 f(u_{a})_{k+2,1}^{(r)} & : = 3 \l \sum_{ n_{1}+ n_{2} = k+2; 2 \leq n_{1}, n_{2} \leq k}\; \sum_{p_{1}+p_{2} + p_{3}= 1} u_{n_{1},p_{1}} u_{n_{2},p_{2}} u_{0,p_{3}} \\
 &\quad + \l \sum_{n_{1}+ n_{2} + n_{3} = k+2; 2 \leq n_{1}, n_{2}, n_{3} \leq k} \; \sum_{p_{1}+p_{2} + p_{3} = 1} u_{n_{1},p_{1}} u_{n_{2},p_{2}} u_{n_{3},p_{3}}.
 \ea
By the induction assumptions, we have for each $n_{1}+ n_{2} = k+2, \ 2 \leq n_{1}, n_{2} \leq k$, that
\ba \label{est-f-ua-k+2-1}
& \|u_{n_{1},p_{1}}(t,\cdot) u_{n_{2},p_{2}}(t,\cdot) u_{0,p_{3}}(t,\cdot)\|_{H^{s-2k}} \\
& \leq C  \|u_{n_{1},p_{1}} (t,\cdot)\|_{H^{s-2k}}  \| u_{n_{2},p_{2}} (t,\cdot) \|_{H^{s-2k}}  \|u_{0,p_{3}}(t,\cdot)\|_{H^{s-2k}} \\
 & \leq  C_0(1+t) ^{n_{1}+n_{2}-2} \leq C_0(1+t) ^{k},
  \ea
 and for each $n_{1}+ n_{2} + n_{3}= k+2, \ 2 \leq n_{1}, n_{2} , n_{3}\leq k$, that
  \ba \label{est-f-ua-k+2-2}
  & \|u_{n_{1},p_{1}}(t,\cdot) u_{n_{2},p_{2}} (t,\cdot) u_{n_{3},p_{3}}(t,\cdot)\|_{H^{s-2k}} \\
& \leq C  \|u_{n_{1},p_{1}} (t,\cdot)\|_{H^{s-2k}}  \| u_{n_{2},p_{2}}(t,\cdot) \|_{H^{s-2k}}  \|u_{n_{3},p_{3}}(t,\cdot)\|_{H^{s-2k}} \\
 & \leq C_0(1+t) ^{n_{1}+n_{2}+n_{3}-3} \leq C_0(1+t)^{k-1}.
 \ea
We then deduce from \eqref{f-ua-k+2-2}--\eqref{est-f-ua-k+2-2}  that
 \ba \label{est-f-ua-k+2}
\| f(u_{a})_{k+2,1}^{(r)}\|_{H^{s-2k}} \leq C_0(1+t) ^{k} .
 \ea
Thus, following the argument in \eqref{g2-duha-de-1}--\eqref{g2-duha-de-2},  by \eqref{f-n1-0-2-k+2} and \eqref{est-f-ua-k+2},  we have
\ba\label{gk+2-duha-de-1}
&\| g_{k+2}(t,\cdot)\|_{H^{s-2(k+2)}}  \\
& \leq   \| g_{k+2}(0,\cdot)\|_{H^{s-2(k+2)}}  + \int_{0}^{t} \frac{1}{2}  \big[\|\d_{tt} g_{k}(t',\cdot )\|_{H^{s-2(k+2)}}  + \|f(u_{a})_{k+2,1}(t',\cdot)\|_{H^{s-2(k+2)}}    \big] \dd t'\\
&\leq C_{0}+ C_{0} \int_{0}^{t} \left[ (1+t')^{k} +  (1+t')^{-\frac{3}{2}} \| g_{k+2}(t',\cdot)\|_{H^{s-2(k+2)}}\right]\dd t'.
\ea
By Gronwall's inequality, we can deduce from \eqref{gk+2-duha-de-1} that
 \ba\label{gk+2-duha-de-2-7}
\| g_{k+2}(t,\cdot)\|_{H^{s-2(k+2)}} & \leq C_0(1+t) ^{k+1}.
\ea
 This is the estimate \eqref{est-g2-1} with $n=k+2$ and $j=0$. The estimates of the time derivatives $\d_{t}^{j} g_{k+2}$  follows from the estimates of $g_{k+2}$, the Schr\"odinger equation \eqref{Schro-g-k+2-even} for $g_{k+2}$ and the induction assumptions.  The computation is rather straightforward and we omit the details.

\medskip

We next  show \eqref{reg-unp-defo} holds for  $n = k+4$. By the second equation in $\eqref{Unp-even-1}$, we have for each $p\geq 3$ that
\ba\label{unp=k+4}
 u_{k+4,p}  = \frac{1}{1-p^{2}} \Big(- f(u_{a})_{k+2,p} - \d^{2}_{t} u_{k,p} + \Delta  u_{k+2,p} -  2ip \d_{t} u_{k+2,p} \Big) .
\ea
By induction assumptions of \eqref{reg-unp-defo} with  $n \leq  k+2$, we have for $p\geq 3$
\ba\label{unp=k+4-1}
\| \d^{2}_{t} u_{k,p} \|_{H^{s - 2 (k+4) + 4}} &= \|  \d_{t}^{2} u_{k,p} \|_{H^{s - 2 k - 4 }}  \leq  C_{0} +  C_{0}(1+  t)^{k-3} \leq C_0(1+t) ^{k-1},\\
\| \Delta  u_{k+2,p} \|_{H^{s - 2 (k+4) + 4}} &= \|   u_{k+2,p} \|_{H^{s - 2 (k+2) + 2 }}  \leq C_0(1+t) ^{k-1} , \\
\| \d_{t} u_{k+2,p} \|_{H^{s - 2 (k+4) + 4}} &= \|   \d_{t} u_{k+2,p} \|_{H^{s - 2 (k+2)  }} \leq C_0(1+t) ^{k-1} .
\ea
Similar as the decomposition \eqref{f-ua-k+2-1} for $f(u_{a})$, using the induction assumptions and the estimate \eqref{gk+2-duha-de-2-7} for $g_{k+2}$ we have shown in the previous step, we can deduce
\ba\label{unp=k+4-2}
 \| f(u_{a})_{k+2,p}  \|_{H^{s - 2 (k+4) + 4}} \leq  C_0(1+t) ^{k+1} .
\ea
By \eqref{unp=k+4}--\eqref{unp=k+4-2} we obtain for each $p\geq 3$ that
\ba\label{unp=k+4-3}
 \| u_{k+4,p}  \|_{H^{s - 2 (k+4) + 4}} \leq  C_0(1+t) ^{k+1} .\nn
\ea
This is \eqref{reg-unp-defo} with  $n = k+4, j=0$. The estimates for the time derivatives $\d_{t}^{j} u_{k+4,p}$ follows from the induction assumptions and the estimates for $g_{k+2}$ we have obtained in the previous step.

By induction principal, we thus complete the proof of the estimates in \eqref{est-g2-1}--\eqref{reg-unp-defo}.
\end{proof}

A direct corollary of Propositions \ref{gn=0-n-odd} and \ref{prop-reg-gn-defo} is the following:
\begin{corollary}\label{prop-reg-Un-defo}
Under the assumptions in Proposition \ref{prop-reg-gn-defo}, there exists a WKB solution $U_{a}$ of the form \eqref{def-wkb} with $U_{n,p} \in C([0,\infty); H^{s-2n}(\R^{3}))$ satisfying the properties stated in Propositions \ref{prop-Un1}, \ref{prop-Unp-p-large}, \ref{gn=0-n-odd},  \ref{ini-Un=0}, \ref{prop-reg-gn-defo}  and Corollary \ref{prop-Unp-gn}. In particular, the amplitudes $U_{n,p}$  satisfy the following estimates for all $t\in (0,\infty)$:
\ba\label{reg-un-defo-1}
&\|U_{n,1}(t,\cdot)\|_{H^{s-2n}}  \leq  C_{0}+C_{0}(1+ t)^{n-1},  \ \  n\in \{0,1,, \cdots, K_{a}\}\cap \Z_{\rm even},\\
& \|U_{K_{a}+2,1}(t,\cdot)\|_{H^{s-2K_{a} -2}}  \leq  C_0 + C_0(1+t) ^{K_{a}-1},\\
&\|U_{n,1}(t,\cdot)\|_{H^{s-2n+1}}  \leq   C_0 + C_0(1+t) ^{n-2} , \ \  n\in \{1,2,\cdots, K_{a}+2\}\cap \Z_{\rm odd},\\
& \|U_{n,p}(t,\cdot)\|_{H^{s-2n+4}}  \leq  C_0 + C_0(1+t) ^{n-3} , \ \   p\geq 3,  \ n\in \{ 1,2,\cdots, K_{a}+2\}\cap \Z_{\rm even},\\
& \|U_{n,p}(t,\cdot)\|_{H^{s-2n+5}}  \leq   C_0 + C_0(1+t) ^{n-4} , \ \  p\geq 3,  \ n\in \{1,2,\cdots, K_{a}+2\}\cap \Z_{\rm odd}.
\ea
In addition, the initial perturbation satisfies
\ba\label{ini-ua-u-defo}
\|U(0,\cdot) - U_{a}(0, \cdot)\|_{H^{s-2 K_{a}-2}} \leq C_{0} \e^{K_{a}+1}.
\ea
\end{corollary}

\begin{proof}

By choosing $g_{_{K_{a}+1}} = g_{_{K_{a}+2}} = 0$ and applying Propositions \ref{prop-Un1}, \ref{prop-Unp-p-large}, \ref{gn=0-n-odd},  \ref{ini-Un=0} and \ref{prop-reg-gn}, \ref{prop-reg-gn-defo} and Corollary \ref{prop-Unp-gn}, we can construct a WKB solution $U_{a}$ of the form \eqref{def-wkb}  with $U_{n,p} \in C([0,\infty); H^{s-2n}(\R^{3}))$ satisfying \eqref{reg-un-defo-1}. The estimate of the initial perturbation \eqref{ini-ua-u-defo} follows from  \eqref{ini-pert-defo} and \eqref{reg-un-defo-1}.
\end{proof}

\subsection{Proof of Theorem \ref{thm:app-defo}}
Now we summarize the results in Propositions \ref{prop-reg-gn}  and \ref{prop-reg-gn-defo} to prove Theorem \ref{thm:app-defo}.  Notice that the form of $U_{a}$ as well as $U_{n,p}$ are the same for both focusing case and defocusing case. The main difference is that for the defocusing case $U_{a}$ can be globally constructed. For the WKB solution $U_{a}$ constructed in Propositions \ref{prop-reg-gn}  and \ref{prop-reg-gn-defo} by imposing $g_{_{K_{a}+1}} = g_{_{K_{a}+2}}  = 0$, there holds $\Phi_{n,p} = 0$ for all $n=-2,-1,0,\cdots, K_{a}$ and all $p\in \Z$. Then there holds
\ba\label{eq-Ua-1}
\d_t \Big(\sum_{n=0}^{K_{a}} \e^{n} U_n\Big)  + \Big(\sum_{n=K_{a}+1}^{K_{a}+2} \e^{n-2} \sum_{p\in \calH_{n}} e^{\frac{ipt}{\e^{2}}}(ip + U_{n,p})\Big) - \frac{1}{\e}A(\nabla) \Big(\sum_{n=0}^{K_{a}+1}\e^{n}  U_n\Big)  \\ +\frac{1}{\e^2} A_0 \Big(\sum_{n=0}^{K_{a}+2} \e^{n}  U_n\Big)    =  \sum_{n=0}^{K_{a}} \e^{n}  F(U_a)_{n}.\nn
\ea
Then $U_{a}$ satisfies the equation
\ba\label{eq-Ua-2}
\d_t  U_{a}  - \frac{1}{\e}A(\nabla)  U_{a}  +\frac{1}{\e^2} A_0  U_{a}  = F(U_a) - \e^{K_{a}+1} R_{\e},\nn
\ea
with
\ba\label{Re-1}
R_{\e}  : = \sum_{p\in \calH_{K_{a}+2}} e^{\frac{ipt}{\e^{2}}} (\d_{t} U_{K_{a}+1,p} + \e \d_{t} U_{K_{a}+2,p}) - A(\nabla) U_{K_{a}+2} - \sum_{n=K_{a}+2}^{3 (K_{a}+2)} \e^{n-K_{a}-1}  F(U_a)_{n}.\nn
\ea

As shown in \eqref{ini-pert-defo}, we have
 \ba\label{ini-pert-defo-2}
U_{a}(0,\cdot) =  U(0,\cdot)  - \e^{K_{a}+2} r_{\e}, \quad r_{\e}: =  - U_{K_{a}+2}(0,\cdot).\nn
 \ea
By using the estimates in Proposition \ref{prop-reg-gn} concerning the focusing case and the estimates in Corollary \ref{prop-reg-Un-defo} concerning the defocusing case,  it is clear that $R_{\e}$ and $r_{\e}$ satisfy the estimate \eqref{est-remainders-defo} and \eqref{est-remainders-fo}, respectively. The proof of Theorem \ref{thm:app-defo} is completed.

\section{Stability of WKB solutions}\label{sec:sta-WKB}
In this section, we show the stability of WKB solutions constructed in Section \ref{sec:wkb} and prove Theorem \ref{thm:sta-defo}. Suppose $\phi, \psi \in H^{s}$ with $s> s_{a}:= 2  K_{a} +4 + \frac{3}{2}$.  We shall introduce the perturbation
\ba\label{def-dotU}
\dot U = \bp \dot w\\ \dot v \\ \dot u \ep: = \frac{U-U_{a}}{\e^{K_{a}+1}}.
\ea
Thus $\dot U$ solves
 \ba\label{eq-dotU}
\left\{ \begin{aligned}
&\d_t \dot U  - \frac{1}{\e}A(\nabla) \dot U  +\frac{1}{\e^2} A_0 \dot U  = \frac{1}{\e^{K_{a}+1}}\big(F(U) -  F(U_{a})\big) +  R_\e ,\\
& \dot U (0,\cdot) = \e r_\e ,\end{aligned}\right.
\ea
where $R_{\e}$ and $r_{\e}$ satisfy \eqref{est-remainders-fo} for the focusing case and satisfy \eqref{est-remainders-defo} for the defocusing case. We first observe that
\ba\label{FU-FUa}
\frac{1}{\e^{K_{a}+1}}\big(f(u) -  f(u_{a})\big) & = \frac{1}{\e^{K_{a}+1}}\big(f(u_{a}+\e^{K_{a}+1} \dot u) -  f(u_{a})\big) \\
 &= \l(3 u_{a}^{2} \dot u + 3 \e^{K_{a}+1} u_{a} \dot u^{2} + \e^{2(K_{a}+1)} \dot u^{3}).
\ea
Throughout this section, we let $\tilde s_{a} := s -2K_{a} - 4 > \frac{3}{2}.$

\subsection{Proof of Theorem \ref{thm:sta-defo}: Part (i)} \label{sec:proof-defo}
For the defocusing case, we see in Theorem \ref{thm:app-defo} or Corollary \ref{prop-reg-Un-defo} that a global-in-time WKB approximate solution $U_{a}$ can be constructed. We shall show that such a WKB solution $U_{a}$ is stable over long time. In this section, we consider high regularity case with $K_{a}\in \Z_{+}\cap \Z_{\rm even}$ and prove Theorem  \ref{thm:sta-defo} (iii).   Since \eqref{eq-dotU} is a symmetric hyperbolic system in $\dot U$, there exists a unique solution $\dot U\in C([0,T_{\e}^{*}); H^{\tds_{a}} )$ with $T_{\e}^{*}$ the existence time.   Then for all $t\in (0, T_{\e}^{*})$, using Duhamel formula gives
\ba\label{duha-dotU}
\dot U (t, \cdot) = \e S(t) r_\e + \int_{0}^{t} S(t-t') \frac{1}{\e^{K_{a}+1}}\big(F(U) -  F(U_{a})\big)(t',\cdot) \dd t' + \int_{0}^{t} S(t-t') R_{\e}(t',\cdot) \dd t',
\ea
where
$$
S(t) := \exp\big(t\big(\frac{1}{\e}A(\nabla) \dot U  - \frac{1}{\e^2} A_0 \big) \big)
$$
 is unitary from $H^{s}$ to $H^{s}$: $\|S(t) U\|_{H^{s}} = \|U\|_{H^{s}} $, for all $s\in \R.$   Thus, by \eqref{FU-FUa}, we know for all $t\in (0,T_{\e}^{*})$ that
\ba\label{est-dotU-defo-1}
&\|\dot U (t, \cdot)\|_{H^{\tilde s_{a}}}  \leq  \e \| r_\e \|_{H^{\tilde s_{a}}} + \int_{0}^{t}\| R_{\e}(t',\cdot)\|_{H^{\tds_{a}}} \dd t'\\
&\quad + \int_{0}^{t} C (\| u_{a}\|_{H^{\tilde s_{a}}} \| u_{a}\|_{L^{\infty}} + \e^{K_{a}+1} \| u_{a} \|_{H^{\tilde s_{a}}} \|\dot u\|_{H^{\tds_{a}}}+\e^{2(K_{a}+1)} \|\dot u\|_{H^{\tds_{a}}}^{2})\|\dot u\|_{H^{\tds_{a}}} (t') \dd t',
\ea
where $r_{\e}$ and $R_{\e}$ satisfy \eqref{est-remainders-defo},   $\tds_{a} : = s_{a} - 2K_{a} - 4 >\frac{3}{2}$.

By Proposition \ref{prop-reg-gn-defo}, we have
 \ba\label{est-ua-3d-1}
 \|u_{a}(t,\cdot)\|_{H^{\tds_{a}}} &  = \|(u_{0} +  \e^{2} u_{2} + \cdots + \e^{K_{a}}  u_{_{K_{a}}} + \e^{K_{a}+2} u_{_{K_{a}+2}})(t,\cdot)\|_{H^{\tds_{a}}}  \\
 & \leq  \|(u_{0}(t,\cdot)\|_{H^{\tds_{a}}}   +  C_{0}\e^{2} (1+t) \big( 1  + \cdots   +    \e^{K_{a}} (1+t)^{K_{a} -1}\big) \\
 & \leq  C_{0}+ C_{0}\e^{2} (1+t)\big(1  + \cdots +   \e^{K_{a}} (1+t)^{K_{a} -1} \big), \ \mbox{for all $t>0$}.
 \ea
Thanks to the decay estimate $\|g_{0}(t,\cdot)\|_{L^{\infty}} \leq  C_0(1+t) ^{-\frac{3}{2}}$ in Proposition \ref{prop-Sch-global}, we have
 \ba\label{est-ua-3d-2}
 \|u_{a}(t,\cdot)\|_{L^{\infty}} &  = \|(u_{0} + \e^{2} u_{2} + \cdots + \e^{K_{a}} u_{_{K_{a}}} + \e^{K_{a}+2} u_{_{K_{a}+2}})(t,\cdot)\|_{L^{\infty}}  \\
 & \leq \|u_{0} \|_{L^{\infty}} + C_{0}\e^{2} (1+t)\big(  1  + \cdots +  \e^{K_{a}} (1+t)^{K_{a} -1}\big) \\
 & \leq C_0(1+t) ^{-\frac{3}{2}} +  C_{0}\e^{2} (1+t) \big(1 + \cdots +    \e^{K_{a}} (1+t)^{K_{a} -1}\big), \ \mbox{for all $t>0$}.
 \ea

Let $ 0<T_{0}\leq 1 $  be determined later. Then from \eqref{est-dotU-defo-1}, \eqref{est-ua-3d-1}, and \eqref{est-ua-3d-2}, we can deduce for all $t\in [0, \min\{T_{\e},\frac{T_{0}}{\e}\})$ that
\ba\label{est-dotU-defo-3d-1}
\|\dot U (t, \cdot)\|_{H^{\tilde s_{a}}} \leq  C_0 (1+t)^{K_{a}}
+ &\int_{0}^{t} \Big(C_0(1+t) ^{-\frac{3}{2}}   + C_0  \e + C_0 \e^{K_{a}+1} \|\dot u\|_{H^{\tds_{a}}}\\
&\qquad  +C\e^{2(K_{a}+1)} \|\dot u\|_{H^{\tds_{a}}}^{2}\Big)\|\dot u\|_{H^{\tds_{a}}}(t') \dd t'.
\ea

For all $t\leq \frac{T_{0}}{\e}$, there holds
\ba\label{est-ini-defo-1}
  C_0 (1+t)^{K_{a}} \leq  C_0 (1+\frac{T_{0}}{\e})^{K_{a}} \leq 2^{K_{a}}  \frac{C_0 }{\e^{K_{a}}}.
\ea
Define
\be\label{def-TT-2}
\widetilde T_{2} :=\sup \Big\{t < \min\big\{T_{\e}^{*}, \frac{T_{0}}{\e}\big\}: \|\dot U \|_{L^\infty(0,t;H^{\tds_{a}})} \leq  e^{3C_0 }2^{K_{a}} \frac{C_0 }{\e^{K_{a}}}\Big\}.
\ee
Then for all $t\leq \widetilde T_{2}$, by \eqref{est-dotU-defo-3d-1} and \eqref{est-ini-defo-1}, there holds
\ba\label{est-dotU-defo-3d-2}
\|\dot U (t, \cdot)\|_{H^{\tilde s_{a}}} \leq 2^{K_{a}}  \frac{C_0 }{\e^{K_{a}}}
+ &\int_{0}^{t}  \Big(C_0 (1+t')^{-\frac{3}{2}}  + C_0 \e +C_0^2    \e e^{3 C_0 }2^{K_{a} } \\
&\qquad + C C_0 ^{2} \e^{2} e^{6C_0 } 2^{2K_{a} }  \Big)\|\dot u\|_{H^{\tds_{a}}}(t') \dd t'.
\ea
Applying Gronwall's inequality to \eqref{est-dotU-defo-3d-2} implies for all $t\leq \widetilde T_{2}$ that
\ba\label{est-dotU-defo-3d-3}
\|\dot U (t, \cdot)\|_{H^{\tilde s_{a}}} & \leq  2^{K_{a}}  \frac{C_0 }{\e^{K_{a}}}e^{\int_{0}^{t} \big( C_0 (1+t')^{-\frac{3}{2}}  + \e \tilde C_0 \big) \dd t'}
 \leq 2^{K_{a}}  \frac{C_0 }{\e^{K_{a}}} e^{2 C_0   +  T_{0}\tilde C_0 }.\nn
\ea
with
$$\tilde C_0 : =  C_0    + C_0^2    e^{3 C_0 }2^{K_{a} } + C C_0^2   \e e^{6 C_0 } 2^{2K_{a} }.$$
Choose $0<T_{0}\leq 1$ small such that
\ba\label{T0-choice-1}
T_{0}\tilde C_0   = T_{0} \big(C_0    + C_0^2    e^{3 C_0 }2^{K_{a} } + C C_0^2   \e e^{6 C_0 } 2^{2K_{a} } \big)\leq  \frac{1}{2} C_0   .
\ea
Then for all  $t\leq \widetilde T_{2}$ there holds
\ba\label{est-dotU-defo-3d-4}
\|\dot U (t, \cdot)\|_{H^{\tilde s_{a}}}  \leq 2^{K_{a}}  \frac{C_0 }{\e^{K_{a}}}  e^{\frac{5}{2} C_0   }.\nn
\ea
Thus, by the definition of $\widetilde T_{2}$ in \eqref{def-TT-2} and continuation argument, we know for each $0<T_{0}\leq 1$ satisfying \eqref{T0-choice-1} there holds 
\ba\label{def-TT-2-new}
\widetilde T_{2} = \min\big\{T_{\e}^{*}, \frac{T_{0}}{\e}\big\}.
\ea

We next claim $T_{\e}^{*} > \frac{T_{0}}{\e}$. Indeed, if $T_{\e}^{*} \leq \frac{T_{0}}{\e}$, then by  \eqref{def-TT-2} and \eqref{def-TT-2-new} one has
$$
\|\dot U \|_{L^\infty(0,t;H^{\tds_{a}})} \leq  e^{3C_0 }2^{K_{a}} \frac{C_0 }{\e^{K_{a}}}, \quad \mbox{for all} \ t < T_{\e}^{*}.
$$
This contradicts to the blow-up criteria:
\ba\label{bp-cre}
\limsup_{t\to T_{\e}^{*}-} \|\dot U \|_{L^\infty(0,t;H^{\tds_{a}})} = +\infty.
\ea

Therefore, by \eqref {est-dotU-defo-3d-1},  for each $0<T_{0}\leq 1$ satisfying \eqref{T0-choice-1} and for all $t\leq \frac{T_{0}}{\e}$, there holds
\ba\label{est-dotU-defo-3d-5}
\|\dot U (t, \cdot)\|_{H^{\tilde s_{a}}} \leq    C_0 (1+t)^{K_{a}}
+ &\int_{0}^{t}  \Big(C_0 (1+t')^{-\frac{3}{2}}  +  C_0  \e +C_0^2    \e e^{3 C_0 }2^{K_{a} } \\
&\qquad + C C_0 ^{2} \e^{2} e^{6C_0 } 2^{2K_{a} } \Big)\|\dot u\|_{H^{\tds_{a}}}(t') \dd t'.
\ea
Then applying Gronwall's inequality to \eqref{est-dotU-defo-3d-5} implies
\ba\label{est-dotU-defo-3d-6}
\|\dot U (t, \cdot)\|_{H^{\tilde s_{a}}} \leq  e^{\frac{5}{2}  C_0 }  C_0 (1+t)^{K_{a}}, \quad \mbox{for all $t\leq \frac{T_{0}}{\e}$}.\nn
\ea
This implies \eqref{stable-defo-3d} and completes the proof of Theorem \ref{thm:sta-defo} (i).

\subsection{Proof of Theorem \ref{thm:sta-defo} (ii)}
The proof of  Theorem \ref{thm:sta-defo} (ii) is rather similar to that of  Theorem \ref{thm:sta-defo} (i) by applying continuation argument. We shall consider a WKB solution of the form \eqref{def-wkb} with $K_{a} = 0$. Thus, by \eqref{U0-form}, \eqref{U1-form} and \eqref{U2-form} with $g_{1} = g_{2} = 0$, we have
$U_{a} = U_{0} + \e U_{1} + \e^{2} U_{2}$
with
\ba\label{Un-Ka=0}
 U_{0} = e^{i\th} g_{0} e_{+} + c.c., \quad U_{1} = e^{i\th} \bp \nabla g_{0}  \\ 0 \\ 0\ep + c.c., \quad  U_{2} = e^{i\th} \bp  0_d \\ \d_{t} g_{0} \\ 0\ep  + e^{3 i\th} \frac{\l g_{0}^{3}}{8} \bp 0_d  \\ 3 i \\ 1\ep + c.c.,
\ea
where $g_{0}$ is the solution to the defocusing cubic  Schr\"odinger equation \eqref{eq-g-cub} and satisfies the estimates in Proposition \ref{prop-Sch-global}. With $K_{a} = 0$, the estimates of the remainders $R_{\e}$ and $r_{\e}$ in \eqref{est-remainders-defo} becomes
\be\label{est-remainders-defo-0}
\|R_{\e} (t,\cdot)\|_{ H^{s-4}} \leq C_0  ,   \quad \|r_{\e} \|_{ H^{s-2}} \leq  C_0 , \quad \mbox{for all $t\in (0,\infty)$}.
\ee
Then, similarly as \eqref{est-dotU-defo-1}, we have for all $t < T_{\e}^{*}$ that
\ba\label{est-dotU-defo-3d-1-0}
\|\dot U (t, \cdot)\|_{H^{s-4}} \leq  \e \| r_\e \|_{H^{s-4}} + &\int_{0}^{t}\Bigl(\| R_{\e}(t',\cdot)\|_{H^{s-4}}
+  C \big((\| u_{a}\|_{H^{s-4}} \| u_{a}\|_{L^{\infty}}  \\
&\quad +  \e \| u_{a} \|_{H^{s-4}} \|\dot u\|_{H^{s-4}}+\e^{2} \|\dot u\|_{H^{s-4}}^{2})\|\dot u\|_{H^{s-4}}\big)(t')\Bigr) \dd t'.
\ea
By \eqref{Un-Ka=0}, we have
\ba\label{est-ua-2d-1-Ka=0}
 \|u_{a}(t,\cdot)\|_{H^{s-4}} &  = \|(u_{0} +  \e^{2} u_{2} )(t,\cdot)\|_{H^{s-4}}  \leq  C_0   + C_0   \e^{2},\\
 \|u_{a}(t,\cdot)\|_{L^{\infty}} &  = \| (u_{0} + \e^{2} u_{2})(t,\cdot)\|_{L^{\infty}}   \leq   C_0(1+t) ^{-\frac{3}{2}} +  C_0   \e^{2}.
  \ea

By \eqref{est-remainders-defo-0}--\eqref{est-ua-2d-1-Ka=0}, we obtain
\ba\label{est-dotU-defo-3d-1-1}
\|\dot U (t, \cdot)\|_{H^{s-4}} \leq  C_0 (1+t)
+& \int_{0}^{t} \Big(C_0  (1+t')^{-\frac{3}{2}}  + C_0  \e^{2} + C_0  \e \|\dot u\|_{H^{s-4}}\\
 &\qquad +C\e^{2} \|\dot u\|_{H^{s-4}}^{2}\Big)\|\dot u\|_{H^{s-4}}(t') \dd t'.
\ea

\medskip

 Let $0< \hat T_{0}\leq 1$ to be determined and let $t\leq \frac{ \hat T_{0}}{\sqrt\e}$.  Then we have
\ba\label{est-ini-0-3d}
C_0  (1+t) \leq \frac{ 2 C_0 }{\sqrt\e}.
\ea
Define
\be\label{def-TT-4}
\widetilde T_{3} :=\sup \Big\{t < \min\big\{T_{\e}^{*}, \frac{\hat T_{0}}{\sqrt\e} \big\}: \|\dot U \|_{L^\infty(0,t;H^{s-4})} \leq  e^{3C_0 } \frac{2 C_0 }{\sqrt\e}\Big\}.
\ee
Then for all $t\leq \widetilde T_{3}$, by  \eqref{est-dotU-defo-3d-1-1}and \eqref{est-ini-0-3d}, there holds
\ba\label{est-dotU-defo-3d-2-0}
\|\dot U (t, \cdot)\|_{H^{ s_{a}-4}} \leq \frac{ 2 C_0 }{\sqrt\e} + \int_{0}^{t}  \Big(C_0  (1+t')^{-\frac{3}{2}}  +  \sqrt \e \hat C_0  \Big)\|\dot u(t',\cdot)\|_{H^{s-4}}\, \dd t',
\ea
with
$$
\hat C_0  : =   C_0   \e^{\frac {3}{2}} + e^{3C_0 } 2  C_0^2  + C \sqrt\e e^{6C_0 } 2^{2}C_0 ^{2}.
$$
Applying Gronwall's inequality to \eqref{est-dotU-defo-3d-2-0} implies for all $t\leq \widetilde T_{3}$ that
\ba\label{est-dotU-defo-3d-3-0}
\|\dot U (t, \cdot)\|_{H^{\tilde s_{a}}} & \leq  \frac{ 2 C_0 }{\sqrt\e} e^{\int_{0}^{t}  \big(C_0  (1+t')^{-\frac{3}{2}}  +  \sqrt \e \hat C_0  \big)  \dd t' }
 \leq \frac{ 2 C_0 }{\sqrt\e}  e^{2 C_0   +  \hat T_{0}\hat C_0  }.\nn
\ea
Choose $0< \hat T_{0}\leq 1$ small such that
\ba\label{T0-choice-1-0}
\hat T_{0}\hat C_0  = \hat T_{0}\big(C_0   \e^{\frac {3}{2}} + e^{3C_0 } 2  C_0^2  + C \sqrt\e e^{6C_0 } 2^{2}C_0 ^{2} \big) \leq  \frac{1}{2} C_0 .
\ea
Then for all  $t\leq \widetilde T_{3}$ with $\hat T_{0}$ satisfying \eqref{T0-choice-1-0},  there holds
\ba\label{est-dotU-defo-3d-4-0}
\|\dot U (t, \cdot)\|_{H^{\tilde s_{a}}}  \leq \frac{ 2 C_0 }{\sqrt\e}    e^{\frac{5}{2} C_0 }.\nn
\ea
By the definition of $\widetilde T_{3}$ in \eqref{def-TT-4} and continuation argument, we know for each $0< \hat T_{0}\leq 1$ satisfying \eqref{T0-choice-1-0} there holds $\widetilde T_{3} =\min\{T_{\e}^{*}, \frac{\hat T_{0}}{\sqrt \e}\}$. Similarly one can employ the blow up criteria \eqref{bp-cre} and further show $T_{\e}^{*} > \frac{\hat T_{0}}{\sqrt \e}$ and   $\widetilde T_{3} = \frac{\hat T_{0}}{\sqrt \e}$.

Therefore, by \eqref{est-dotU-defo-3d-1-0},  for each $0< \hat T_{0}\leq 1$ satisfying \eqref{T0-choice-1} and for all $t\leq \frac{\hat T_{0}}{\sqrt\e}$, there holds
\ba\label{est-dotU-defo-3d-5-0}
&\|\dot U (t, \cdot)\|_{H^{\tilde s_{a}}} \leq    C_0  (1+t) + \int_{0}^{t}  \Big(C_0  (1+t')^{-\frac{3}{2}}  +  \sqrt \e \hat C_0  \Big)\|\dot u(t',\cdot)\|_{H^{s-4}} \,\dd t'.
\ea
Applying Gronwall's inequality to \eqref{est-dotU-defo-3d-5-0} and using \eqref{T0-choice-1-0} implies
\ba\label{est-dotU-defo-3d-6-0}
\|\dot U (t, \cdot)\|_{H^{\tilde s_{a}}} \leq C_0 \, e^{\frac{5}{2}  C_0 }  (1+t), \quad \mbox{for all $t\leq \frac{ \hat T_{0}}{\sqrt\e}$}.\nn
\ea
This implies \eqref{stable-defo-3d-0} and completes the proof of  Theorem  \ref{thm:sta-defo} (ii).

\subsection{Proof of Theorem \ref{thm:sta-defo}: Part (iii)}

We finally consider the focusing case $\l<0$. Recll that $T^{*}$ is the existence time of \eqref{eq-g-cub}. Since \eqref{eq-dotU} is a symmetric hyperbolic system in $\dot U$, there exists a unique solution $\dot U\in C([0,T_{\e}^{*}); H^{\tds_{a}} )$ with $T_{\e}^{*}$ the existence time.   Then for all $t\in (0,\min\{T^{*}, T_{\e}^{*}\})$, 
\ba\label{est-dotU-fo-1}
\|\dot U (t, \cdot)\|_{H^{\tilde s_{a}}} \leq & \e \| r_\e \|_{H^{\tilde s_{a}}} + \int_{0}^{t}\| R_{\e}(t',\cdot)\|_{H^{\tds_{a}}} \dd t' \\
&+ \int_{0}^{t} C\big((\| u_{a}\|_{H^{\tilde s_{a}}}^{2}  + \e^{K_{a}+1} \| u_{a} \|_{H^{\tilde s_{a}}} \|\dot u\|_{H^{\tds_{a}}}+\e^{2(K_{a}+1)} \|\dot u\|_{H^{\tds_{a}}}^{2})\|\dot u\|_{H^{\tds_{a}}}\big)(t') \dd t'.
\ea
By  \eqref{est-remainders-fo}  in Theorem \ref{eq-WKB-fo} concerning the estimates of  $R_{\e}$ and $r_{\e}$, by Proposition \ref{prop-reg-gn} concerning the estimates of $U_{a}$, we can deduce from \eqref{est-dotU-fo-1} that
\ba\label{est-dotU-fo-2}
\|\dot U (t, \cdot)\|_{H^{\tilde s_{a}}} \leq& \e C_{0} + t\, D(t) \\
&+ \int_{0}^{t}\big(( D(t')  +  D(t') \e^{K_{a}+1} \|\dot u\|_{H^{\tds_{a}}}+ C \e^{2(K_{a}+1)} \|\dot u\|_{H^{\tds_{a}}}^{2})\|\dot u\|_{H^{\tds_{a}}}\big)(t') \dd t'.
\ea

Let $T<T^*$ be an arbitrary positive number.  We shall show $\liminf_{\e\to 0} T_{\e}^{*} \geq T^{*}$. To do so,  it is sufficient to show there exists $\e_0>0$ such that $ T_\e^* \geq  T$ for all $0<\e<\e_0$. Applying Gronwall's inequality to \eqref{est-dotU-fo-2} implies for all $ t <\min\{T,T_{\e}^{*}\}$ that
\ba\label{est-dotU-fo-3}
\|\dot U \|_{L^{\infty}(0,t;H^{\tilde s_{a}})} \leq \big( \e C_0 + t D(t)\big) e^{ t D(t) (1  +   \e^{K_{a}+1} \|\dot u\|_{L^{\infty}(0,t;H^{\tilde s_{a}})})
 + t C \e^{2(K_{a}+1)} \|\dot u\|^{2}_{L^{\infty}(0,t;H^{\tilde s_{a}})}}.
\ea
Define
 $M(T):= \big( C_0 + T D (T)\big)e^{2 T (D(T)+C)},$
and
\be\label{def-TT-1}
\widetilde T_{1} :=\sup \big\{t< \min\{T,T_{\e}^{*}\}: \|\dot U \|_{L^\infty(0,t;H^{\tds_{a}})} \leq  M(T)\big\}.
\ee
If $\widetilde T_{1} < \min\{T,\widetilde T_\e^*\}$, we then deduce from \eqref{est-dotU-fo-3} for all $t\leq \widetilde T_{1} $, that
\ba\label{est-dotU-fo-4}
 \|\dot U \|_{L^{\infty}(0,t;H^{\tilde s_{a}})} \leq  \big( C_0 + T D(T)\big) e^{T D(T)\left(1  +  \e^{K_{a}+1} M(T)\right) + T C \e^{2(K_{a}+1)}  M(T)^{2}}.\nn
\ea
Let $\e_{0}>0$ small such that
$$\e_{0}^{K_{a}+1} M(T) + \e_{0}^{2(K_{a}+1)}  M(T)^{2} \leq \frac{1}{2}.$$
Then for any $0<\e<\e_0$, there holds
\ba\label{est-dotU-fo-5}
 \|\dot U \|_{L^{\infty}(0,t;H^{\tilde s_{a}})} \leq \big( C_0 + T D(T)\big) e^{ \frac{3}{2}T D(T)  +  \frac{1}{2}T C}.\nn
\ea
By the definition of $\widetilde T_{1}$ in \eqref{def-TT-1}, using continuation argument implies that
\be\label{exist-time-1}\widetilde T_{1} \geq  \min\{T,  T_\e^*\},\quad \mbox{for each $0<\e<\e_0$}.\nn
\ee
Thus, the existence time $T_{\e}^{*} \geq T$ for each $0<\e<\e_0$. Since $T<T^*$ is an arbitrary number, we obtain  $\liminf_{\e\to 0} T_{\e}^{*} \geq T^{*}$ which is exactly \eqref{exist-time-fo}. Moreover, for each $t<\min\{T^{*}, T_{\e}^{*}\}$ there holds
$
\|\dot U(t,\cdot)\|_{H^{\tds_{a}}} \leq D(t).
$
We thus complete the proof of Theorem \ref{thm:sta-defo} (iii) by  noticing
$
U  - U_{a} = \e^{K_{a}+1} \dot U.
$

\section{Error estimates in the nonrelativistic regime} \label{sec:sta-KG-S}
In this section, we will finally prove Theorems \ref{thm2} and \ref{thm4} which are actually corollaries of Theorems \ref{thm:sta-defo}. 

Indeed, Theorem \ref{thm4} is a direct corollary of Theorem \ref{thm:sta-defo}, Theorem \ref{thm2} (ii) is direct corollary of Theorem \ref{thm:sta-defo} (ii). 

 Theorem \ref{thm2} (i) is  a corollary of Theorem \ref{thm4} (i). Since $s>8+\frac{d}{2} = 2 K_{a} + 4 + \frac{d}{2}$ with $K_{a} = 2$, then applying Theorem \ref{thm4} (i) gives 
 \ba\label{thm2-pf-3d}
\|(u- u_{0} -\e^{2} u_{2} - \e^{4} u_{4})(t,\cdot)\|_{H^{s-8}(\R^{3})}  \leq C_0(1+t) ^{2}\e^{3}, \quad   \mbox{for all $ t \leq \frac{T_{0}}{\e}$,}
\ea
where $u_{0}, u_{2}, u_{4}$ satisfy  \eqref{unp-form-thm4} and \eqref{est-ua-thm4}.  We deduce from \eqref{thm2-pf-3d} that
 \ba\label{thm2-pf-3d-1}
\|(u- u_{0})(t,\cdot)\|_{H^{s-8}} &\leq  \e^{2}\| u_{2} (t,\cdot)\|_{H^{s-8}}  + \e^{4}\| u_{4} (t,\cdot)\|_{H^{s-8}}   + C_0(1+t) ^{2}\e^{3}\\
&\leq  C_0  (1+t) \e^{2}\big(1+  \e^{2}   +  (1+t) \e\big)\\
&\leq C_0  (1+t) \e^{2}, \quad \mbox{for all $ t \leq \frac{T_{0}}{\e}$,}\nn
\ea
which is exactly \eqref{stable-defo-thm2-3d}.


\section*{Acknowledgments}
This work was partially supported by the Ministry of Education of Singapore under its  AcRF Tier 2 funding MOE-T2EP20122-0002 (A-8000962-00-00) (W. Bao),  by NSF of China under Grant 12171235 (Y. Lu), and by NSF of China under Grants  12171010 and 12288101 (Z. Zhang). Part of the work was done when the first two authors were visiting the Institute for Mathematical Sciences at the National University of Singapore in February 2023. The authors would like to thank the anonymous referees for the valuable comments and suggestions.


\end{document}